\documentclass[reqno, 12pt]{article}

\usepackage{enumerate}
\usepackage{needspace}
\usepackage{latexsym}
\usepackage{amsmath}
\usepackage{amsfonts}
\usepackage{amsthm}
\usepackage{amssymb,mathtools}
\usepackage{newlfont}
\usepackage{graphics}
\usepackage{color}
\usepackage{float}
\usepackage{diagbox}
\usepackage{tocloft}
\usepackage{titlesec}
\usepackage{booktabs,array,longtable,enumitem,microtype}
\usepackage{extpfeil}
\usepackage{centernot}

\usepackage[linesnumbered,ruled,vlined]{algorithm2e}
\usepackage{url}

\usepackage[margin=1in]{geometry}

\usepackage[T1]{fontenc}
\usepackage{lmodern}
\usepackage{microtype}
\usepackage{rotating}
\usepackage{multirow}
\usepackage{extarrows}
\usepackage[sort,compress,numbers]{natbib}
\usepackage[utf8]{inputenc}
\usepackage{xcolor}
\usepackage{listings}
\usepackage{aliascnt}
\usepackage[pagebackref,colorlinks=true,linkcolor=blue,citecolor=red,urlcolor=blue]{hyperref}
\usepackage[nameinlink,noabbrev,capitalize]{cleveref}
\numberwithin{equation}{section}

\setlist[enumerate]{label=\textup{(\roman*)},leftmargin=*}
\newtheorem{theorem}{Theorem}[section]
\newaliascnt{proposition}{theorem}
\newtheorem{proposition}[proposition]{Proposition}
\aliascntresetthe{proposition}
\newaliascnt{lemma}{theorem}
\newtheorem{lemma}[lemma]{Lemma}
\aliascntresetthe{lemma}
\newaliascnt{corollary}{theorem}
\newtheorem{corollary}[corollary]{Corollary}
\aliascntresetthe{corollary}

\theoremstyle{definition}
\newaliascnt{definition}{theorem}
\newtheorem{definition}[definition]{Definition}
\aliascntresetthe{definition}
\newaliascnt{remark}{theorem}

\aliascntresetthe{remark}
\newaliascnt{example}{theorem}
\newtheorem{example}[example]{Example}
\aliascntresetthe{example}
\newaliascnt{conjecture}{theorem}

\aliascntresetthe{conjecture}
\newaliascnt{question}{theorem}

\aliascntresetthe{question}
\newaliascnt{problem}{theorem}

\aliascntresetthe{problem}
\newaliascnt{alg}{theorem}

\aliascntresetthe{alg}

\allowdisplaybreaks[4]

\SetKwInput{KwInput}{Input}                
\SetKwInput{KwOutput}{Output}              

\newcommand{\Ctrans}{\mathcal{C}}
\newcommand{\Lfunc}{\mathcal{L}}
\newcommand{\Afunc}{\mathcal{A}}

\newcommand{\trans}{\mathsf{T}}

\newcommand{\Ring}{R}

\title{Proof of Barry's Four Hankel Determinant Conjectures}

\author{Feihu Liu$^{\color{blue} \dag}$, Ying Wang$^{\color{blue} \ddag}$, and Zihao Zhang$^{\color{blue} \S}$
\\[2mm]
{\small $^{\color{blue} \dag}$ Center for Combinatorics, LPMC}\\[-0.8ex]
{\small Nankai University, Tianjin 300071, P.R.~China}\\
{\small $^{\color{blue} \ddag}$ School of Mathematics and Statistics,}\\[-0.8ex]
{\small North China University of Water Resources and Electric Power, Zhengzhou, 450045, P.R.~China}\\
{\small $^{\color{blue} \S}$ School of Mathematics and Statistics}\\[-0.8ex]
{\small Beijing Institute of Technology, Beijing 102400, P.R.~China}\\
{\small {\color{blue} $^\dag$} Email address: liufeihu7476@163.com}\\
{\small {\color{blue} $^\ddag$} Email address: wangying2019@ncwu.edu.cn}\\
{\small {\color{blue} $^\S$} Email address: zihao-zhang@foxmail.com}\\
}

\date{\today}

\begin{document}

\maketitle

\begin{abstract}
Barry introduced a central transform of integer sequences and proposed four conjectures concerning the Hankel transforms of central transform of four rational families. We prove these four conjectures. The proofs are unified within a common algebraic framework: we interpret the Hankel determinants as Gram determinants and use a basis of shifted monic Chebyshev polynomials to reveal the finite-band structure of the associated Gram matrices.
\end{abstract}

\noindent
\begin{small}
\end{small}

\noindent
\begin{small}
\end{small}


\section{Introduction}\label{sec:introduction}

Throughout this section, \(R\) denotes a commutative \(\mathbb{Q}\)-algebra. All power-series manipulations are formal.

We introduce a powerful definition in enumerative combinatorics. It was introduced in~\cite{Shapiro1991}.
\begin{definition}\label{def:riordan-array}
Let \(d(x),f(x)\in\Ring[[x]]\), \(f(0)=0\), and \(f'(0)\) is a unit.  The \emph{Riordan array} \((d(x),f(x))\) is the infinite lower-triangular matrix whose \((n,k)\)-entry is
\[r_{n,k}=[x^n]d(x)f(x)^k,\qquad n,k\ge0.\]
Furthermore, if \(d(0)\) is a unit, then \((d(x),f(x))\) is called a \emph{proper Riordan array}.
\end{definition}

It is clear that a proper Riordan array is invertible.
The elements $f(x)$ is compositional invertible, with inverse $\overline{f}(x)$, where $f(\overline{f}(x))=x$ and $\overline{f}(f(x))=x$.
Let $\mathcal{R}$ be the set of all proper Riordan arrays. Then $\mathcal{R}$ forms a group, which we refer to as the \emph{Riordan group} \cite[Theorem 3.2]{LW. SSBCHMW}.
Its multiplication law is
\begin{align*}
(d(x),f(x))(e(x),h(x))=\bigl(d(x)e(f(x)),\,h(f(x))\bigr).
\end{align*}
The inverse is given by 
$$(d(x),f(x))^{-1}=\left(\frac{1}{d(\overline{f}(x))},\overline{f}(x) \right).$$
For more results on the Riordan group and applications, see \cite{LW. SSBCHMW,Cheon2017,Luzon2010,LuzonMoron2008}.

For a series \(g(x)\) with \(g(0)=1\), the \emph{Appell Riordan array}
\((g(x),x)\) has inverse $(g(x),x)^{-1}=\left(\frac1{g(x)},x\right)$.
The \emph{binomial Riordan array} is $\mathcal{P}=\left(\frac1{1-x},\frac{x}{1-x}\right)$.
Hence
\begin{equation}\label{eq:constructed-riordan}
\mathcal{P}\left(\frac{1}{g(x)},x\right)=\left(\frac{1}{1-x}\frac{1}{g\!\left(\frac{x}{1-x}\right)},\frac{x}{1-x}\right).
\end{equation}

Barry's \emph{central transform} is obtained by taking the inverse of an Appell Riordan array, multiplying by the binomial Riordan array, and extracting the entries in positions \((2n,n)\).
\begin{definition}[Central transform]\label{def:central-transform}
Let \(g(x)\in\Ring[[x]]\) satisfy \(g(0)=1\).  Its central transform is
\begin{align*}
\Ctrans(g)(z)=\sum_{n\ge0}\left[x^{2n}\right]\frac1{(1-x)g\!\left(\frac{x}{1-x}\right)}\left(\frac{x}{1-x}\right)^n z^n.
\end{align*}
Thus the coefficient of \(z^n\) is the \((2n,n)\)-entry of the
Riordan array in \cref{eq:constructed-riordan}.
\end{definition}

Let
\begin{align*}
c(z)=\frac{1-\sqrt{1-4z}}{2z} =1+z+2z^2+5z^3+\cdots
\end{align*}
be the Catalan generating function. The series \(c(z)\)
satisfies $c(z)=1+zc(z)^2$.

\begin{proposition}{\em \cite[Proposition 6]{Barry2020}} \label{prop:C-closed-form}
For every \(g(x)\in\Ring[[x]]\) with \(g(0)=1\), a closed form of the central transform is given by 
\begin{equation}\label{eq:C-closed-form}
\Ctrans(g)(z)=\frac1{\sqrt{1-4z}\,g\!\bigl(zc(z)^2\bigr)}.
\end{equation}
\end{proposition}

\begin{definition}[Hankel transform]\label{def:hankel-transform}
For a sequence \((\mu_n)_{n\ge0}\) in a commutative ring, its \(n\)-th \emph{Hankel determinant} is
\[h_n=\det H_n,\qquad H_n=(\mu_{i+j})_{0\le i,j\le n}.\]
The sequence \((h_n)_{n\ge0}\) is called the \emph{Hankel transform} of \((\mu_n)_{n\ge0}\).
\end{definition}

Hankel determinants have drawn significant attention from scholars in combinatorics and number theory. Although the Hankel determinants for particular sequences have been investigated in several works \cite{chien2022hankel,Elouafi,Mu-Wang,Mu-Wang-Yeh,QH-Hou,Wang-Xin}, considerable effort has also been directed toward the development of continued fraction techniques for their evaluation \cite{Krattenthaler05, FlajoletDM,Gessel-Xin,Sulanke-Xin,HanGuoNiu,Wang-Xin-Zhai}. 

Given a \(g(x)\in\Ring[[x]]\) with \(g(0)=1\), we write $M_g(z)=\Ctrans(g)(z)=\sum_{n\ge0}\mu_n(g)z^n$,
and define
\[ h_n(g)=\det\bigl(\mu_{i+j}(g)\bigr)_{0\le i,j\le n},\qquad H_g(x)=\sum_{n\ge0}h_n(g)x^n.
\]
In \cite{Barry2020}, Barry's Conjectures~12, 17, 24, and~27 concern the rationality and explicit form of \(H_g(x)\) for four elementary rational families. The main contribution of this paper is to prove the following four theorems, namely, Barry's four conjectures.

\begin{theorem}{\em (\cite[Conjecture~12]{Barry2020})}\label{thm:c12}
Let \(a,b\) belong to a field of characteristic zero, and set
\[g_{a,b}(x)=\frac{1+ax}{1+bx}.
\]
Then
\begin{align*}
H_{g_{a,b}}(x)=\frac{1-b(a+b)x}{1-2(1+ab)x+(a+b)^2x^2}.
\end{align*}
Equivalently,
\begin{align*}
 h_0=1,\quad h_1=2+ab-b^2,\quad h_n=2(1+ab)h_{n-1}-(a+b)^2h_{n-2}\qquad(n\ge2).
\end{align*}
\end{theorem}

We note that Conjectures 17 and 24 in Barry's original paper have typos. We correct them and give the correct results as follows.

\begin{theorem}{\em (\cite[Conjecture~17]{Barry2020})}\label{thm:c17}
Let \(a,b\) belong to a field of characteristic zero, and set
\[
  g_{a,b}(x)=\frac{1+ax}{1-bx^2}.
\]
Then
\begin{equation}\label{c17:eq:main-H}
  H_{g_{a,b}}(x)=\frac{1-3bx+b^2(2+b)x^2-b^4x^3}{1-2(1+b)x+\bigl(a^2+4b-2a^2b+2b^2+a^2b^2\bigr)x^2-2b^2(1+b)x^3+b^4x^4}.
\end{equation}
Equivalently, if $\sigma=a(1-b)$, then
\begin{align}
  & h_0=1,\quad h_1=2-b,\quad h_2=4-2b-2b^2+b^3-\sigma^2, \nonumber 
  \\& h_3=8-4b-6b^2+2b^3+b^4-(b+4)\sigma^2,\label{c17:eq:h0123-intro}
\end{align}
and, for \(n\ge4\),
$$h_n=2(1+b)h_{n-1}-\bigl(\sigma^2+4b+2b^2\bigr)h_{n-2}+2b^2(1+b)h_{n-3}-b^4h_{n-4}.$$
\end{theorem}

\begin{theorem}{\em (\cite[Conjecture~24]{Barry2020})}\label{thm:c24}
Let \(a\) belong to a field of characteristic zero, and set
\begin{equation}\label{c24:eq:ga-intro}
  g_a(x)=\frac{1+ax}{1-x^3}.
\end{equation}
Then
\begin{equation}\label{c24:eq:main-result}
  H_{g_a}(x)=\frac{1+(1-a)x+x^2-x^3}{1-(a+1)x+\bigl((a+1)^2-1\bigr)x^2-(a+1)x^3+x^4}.
\end{equation}
Equivalently,
\begin{align}\label{c24:eq:h0123-main}
  h_0=1,\quad h_1=2, \quad  h_2=3-a^2,\quad h_3=3-3a^2-a^3,
\end{align}
and, for \(n\ge4\),
\begin{equation}\label{c24:eq:h-recurrence-main}
h_n=(a+1)h_{n-1}-\bigl((a+1)^2-1\bigr)h_{n-2}+(a+1)h_{n-3}-h_{n-4}.
\end{equation}
\end{theorem}

\begin{theorem}{\em (\cite[Conjecture~27]{Barry2020})}\label{thm:c27}
Let \(r,s\) belong to a field of characteristic zero, and set
\[g_{r,s}(x)=\frac{1-(r-2)x+x^2}{1-sx-x^2}.
\]
Then
\begin{equation}\label{c27:eq:main-generating-function}
H_{g_{r,s}}(x)=\frac{1-s(r+s-2)x}{1+2s(2-r)x+4s^2x^2}.
\end{equation}
Equivalently,
\begin{align}
 h_0=1, \quad h_1=s(r-s-2),\quad h_n=2s(r-2)h_{n-1}-4s^2h_{n-2},\qquad(n\ge2).\label{c27:eq:h-recurrence-intro}
\end{align}
\end{theorem}

All four proofs are polynomial identities in their parameters.  We therefore work over the relevant polynomial ring in each application. The common mechanism behind the proofs is that a shifted monic Chebyshev basis converts the associated Hankel form into a banded or nearly banded Gram form. 

The paper is organized as follows. \Cref{sec:hankel-gram} records the Gram-determinant interpretation of Hankel determinants.
Moreover, we introduce the shifted Chebyshev basis and prove some related lemmas that are needed.
The four families are treated in \cref{sec:c12,sec:c17,sec:c24,sec:c27}, respectively.

\section{Gram matrix and a shifted Chebyshev basis}\label{sec:hankel-gram}

\begin{definition}[Gram matrix]
Let $\mathcal{H}$ be an inner product space over $\mathbb{R}$ (or $\mathbb{C}$), and let 
$\{v_1, v_2, \dots, v_n\} \subset \mathcal{H}$ be a finite set of vectors.
The \emph{Gram matrix} (also called the \emph{Gramian}) of $\{v_1, \dots, v_n\}$ 
is the $n \times n$ matrix $\mathbf{G} = (G_{ij})$ whose entries are given by
\[G_{ij} = \langle v_i, v_j \rangle, \qquad i,j = 1,\dots,n.\]
\end{definition}

Define the \(\Ring\)-linear functional
\begin{equation}\label{eq:moment-functional}
\Lfunc:\Ring[X]\longrightarrow\Ring,\qquad\Lfunc(X^n)=\mu_n\quad(n\ge0).
\end{equation}
Then $\Lfunc(X^iX^j)=\mu_{i+j}$, so \(H_n\) is the Gram matrix of the ordered monomial basis
\(1,X,\ldots,X^n\) with respect to the symmetric bilinear form $\langle f,g\rangle_{\Lfunc}=\Lfunc(fg)$.

The following elementary observation allows us to replace the monomial basis by any monic polynomial basis.
\begin{lemma}\label{lem:monic-change-basis}
Let \(p_0(X),\ldots,p_n(X)\in\Ring[X]\) satisfy
\[p_j(X)=X^j+\text{terms of degree less than }j.\]
Set $\Gamma_n=(\Lfunc(p_ip_j))_{0\le i,j\le n}$. Then $\det\Gamma_n=h_n$.
\end{lemma}
\begin{proof}
Write $p_i(X)=\sum_{r=0}^i u_{i,r}X^r$  and let \(U_n=(u_{i,r})_{0\le i,r\le n}\).  Since \(p_i\) is monic of
degree \(i\), the matrix \(U_n\) is lower triangular and all its diagonal entries are \(1\).  Hence $\det U_n=1$.
For \(0\le i,j\le n\),
\begin{align*}
(\Gamma_n)_{i,j}=\Lfunc\left(\sum_{r=0}^i u_{i,r}X^r\sum_{s=0}^j u_{j,s}X^s\right)=\sum_{r=0}^i\sum_{s=0}^j u_{i,r}u_{j,s}\Lfunc(X^{r+s})
=\sum_{r=0}^i\sum_{s=0}^j u_{i,r}\mu_{r+s}u_{j,s}.
\end{align*}
Therefore $\Gamma_n=U_nH_nU_n^{\mathsf T}$.
Taking determinants gives $\det\Gamma_n=(\det U_n)^2\det H_n=h_n$. This completes the proof.
\end{proof}

Let \(T_n(Y)\) be the Chebyshev polynomial of the first kind (see \cite[A053120]{Sloane23} or \cite{TJ.Rivlin}), defined by
\begin{equation}\label{eq:chebyshev-recurrence}
T_0(Y)=1,\qquad T_1(Y)=Y, \qquad T_{n+1}(Y)=2YT_n(Y)-T_{n-1}(Y) \quad(n\ge 1).
\end{equation}

\begin{lemma}\label{lem:chebyshev-generating-function}
The generating function of the Chebyshev polynomials is given by 
\begin{equation}\label{eq:chebyshev-generating-function}
\sum_{n\ge 0}T_n(Y)u^n=\frac{1-Yu}{1-2Yu+u^2}.
\end{equation}
\end{lemma}
\begin{proof}
Let $S(Y,u)=\sum_{n\ge 0}T_n(Y)u^n$. Then
\begin{align*}
(1-2Yu+u^2)S(Y,u)&=T_0(Y)+\bigl(T_1(Y)-2YT_0(Y)\bigr)u 
\\ &\qquad +\sum_{n\ge 2}\bigl(T_n(Y)-2YT_{n-1}(Y)+T_{n-2}(Y)\bigr)u^n.
\end{align*}
The sum over \(n\ge 2\) vanishes by \cref{eq:chebyshev-recurrence}. Since
\(T_0(Y)=1\) and \(T_1(Y)=Y\), the remaining expression is $1+(Y-2Y)u=1-Yu$.
This prove \cref{eq:chebyshev-generating-function}.
\end{proof}

Define shifted polynomials \(q_n(X)\) by
\begin{equation}\label{eq:q-definition}
 q_0(X)=1,
 \qquad
 q_n(X)=2T_n\!\left(\frac{X-2}{2}\right)
 \quad(n\ge 1).
\end{equation}

For the multiplication identities below, let \(r\in R\) be arbitrary.

\begin{lemma}\label{lem:q-basic-properties}
The polynomials \(q_n\) satisfy the following statements.
\begin{enumerate}
\item Each \(q_n\) is monic of degree \(n\).

\item Their generating function is
\begin{equation}\label{eq:q-generating-function}
 \sum_{n\ge 0}q_n(X)u^n
 =
 \frac{1-u^2}{(1+u)^2-Xu}.
\end{equation}

\item If \(m,n\ge 1\) and \(m\neq n\), then
\begin{equation}\label{eq:q-product-off-diagonal}
 q_m(X)q_n(X)
 =
 q_{m+n}(X)+q_{|m-n|}(X).
\end{equation}
If \(n\ge 1\), then
\begin{equation}\label{eq:q-product-diagonal}
 q_n(X)^2=q_{2n}(X)+2q_0(X).
\end{equation}

\item The multiplication recurrences are
\begin{align}
 (X-r)q_0&=q_1+(2-r)q_0,\label{eq:Xr-q0}\\
 (X-r)q_1&=q_2+(2-r)q_1+2q_0,\label{eq:Xr-q1}\\
 (X-r)q_n&=q_{n+1}+(2-r)q_n+q_{n-1}
 \quad(n\ge 2).\label{eq:Xr-qn}
\end{align}
\end{enumerate}
\end{lemma}
\begin{proof}
For \(n\ge 1\), the leading coefficient of \(T_n(Y)\) is \(2^{n-1}\). This follows by induction from \cref{eq:chebyshev-recurrence}: it is true for \(T_1(Y)=Y\), and if the leading coefficient of \(T_n\) is \(2^{n-1}\), then the term \(2YT_n(Y)\) contributes leading coefficient \(2^n\) to \(T_{n+1}\), while \(T_{n-1}\) has lower degree. Therefore $2T_n\!\left(\frac{X-2}{2}\right)$ has leading coefficient $2\cdot 2^{n-1}\cdot 2^{-n}=1$. Thus \(q_n\) is monic of degree \(n\). The assertion is immediate for
\(q_0=1\).

Set $Y=\frac{X-2}{2}$. By \cref{eq:chebyshev-generating-function},
\begin{align*}
\sum_{n\ge 0}q_n(X)u^n &= 1+2\sum_{n\ge 1}T_n(Y)u^n= 2\sum_{n\ge 0}T_n(Y)u^n-1\\
&=\frac{2(1-Yu)}{1-2Yu+u^2}-1=\frac{1-u^2}{1-2Yu+u^2}.
\end{align*}
Since \(2Y=X-2\), $1-2Yu+u^2=(1+u)^2-Xu$, which proves \cref{eq:q-generating-function}.

To prove the product identities, introduce an indeterminate \(w\) and put $Y=\frac{w+w^{-1}}{2}$.
An induction based on \cref{eq:chebyshev-recurrence} gives
\begin{equation}\label{eq:chebyshev-laurent-representation}
T_n\!\left(\frac{w+w^{-1}}{2}\right)=\frac{w^n+w^{-n}}{2}.
\end{equation}
Indeed, both sides agree for \(n=0,1\), and both satisfy the same recurrence.
Consequently,
\begin{align*}
2T_m(Y)T_n(Y)&=\frac{(w^m+w^{-m})(w^n+w^{-n})}{2}=\frac{w^{m+n}+w^{-(m+n)}}{2}+\frac{w^{m-n}+w^{-(m-n)}}{2}
\\&=T_{m+n}(Y)+T_{|m-n|}(Y).
\end{align*}
The substitution map
\[\mathbb{Q}[Y]\longrightarrow\mathbb{Q}[w,w^{-1}], \qquad Y\longmapsto\frac{w+w^{-1}}{2},
\]
is injective: if a nonzero polynomial has degree \(d\), then its Laurent
expansion after substitution has a nonzero coefficient of \(w^d\). Hence the
preceding identity is a polynomial identity in \(Y\). Multiplying by \(2\)
and using \cref{eq:q-definition} gives
\cref{eq:q-product-off-diagonal} when \(m\neq n\). When \(m=n\), the term
\(T_0(Y)=1\) produces \(2q_0\), and we obtain
\cref{eq:q-product-diagonal}.

Finally, \(q_1=X-2\), and $q_2=2T_2\!\left(\frac{X-2}{2}\right)=(X-2)^2-2$.
Thus $(X-2)q_0=q_1$, $(X-2)q_1=q_2+2q_0$.
For \(n\ge 2\), the Chebyshev recurrence gives $(X-2)q_n=q_{n+1}+q_{n-1}$.
Adding \((2-r)q_n\) to each of these identities proves \cref{eq:Xr-q0,eq:Xr-q1,eq:Xr-qn}.
\end{proof}

The next proposition is valid for every central transform.

\begin{proposition}\label{prop:q-moment-general}
Let \(g(0)=1\), $\Ctrans(g)(z)=\sum_{n\ge0}\mu_nz^n$, and \(\Lfunc(X^n)=\mu_n\).  For the polynomials \(q_n\) in \cref{eq:q-definition}, we have 
\begin{equation}\label{eq:q-moment-general}
\sum_{n\ge0}\Lfunc(q_n)u^n=\frac1{g(u)}.
\end{equation}
\end{proposition}
\begin{proof}
Apply \(\Lfunc\), coefficientwise in \(u\), to
\cref{eq:q-generating-function}.  Since the denominator has constant
term \(1\), all expressions are well-defined formal power series.  We
obtain
\begin{align}
\sum_{n\ge0}\Lfunc(q_n)u^n
&=\Lfunc\left(
    \frac{1-u^2}{(1+u)^2-Xu}
  \right)
=\frac{1-u^2}{(1+u)^2}
    \Lfunc\left(
    \frac1{1-\dfrac{Xu}{(1+u)^2}}
    \right)\nonumber\\
&=\frac{1-u^2}{(1+u)^2}
    \Lfunc\left(
    \sum_{k\ge0}X^k
    \left(\frac{u}{(1+u)^2}\right)^k
    \right)
=\frac{1-u^2}{(1+u)^2}
    \sum_{k\ge0}\mu_k
    \left(\frac{u}{(1+u)^2}\right)^k\nonumber\\
&=\frac{1-u}{1+u}
    \Ctrans(g)\!\left(\frac{u}{(1+u)^2}\right).
\label{eq:q-moment-before-substitution}
\end{align}
Set $z=\frac{u}{(1+u)^2}$.
Then $1-4z=\frac{(1-u)^2}{(1+u)^2}$.
Both formal square roots below have constant term \(1\); hence
\begin{equation}\label{eq:sqrt-u-substitution}
  \sqrt{1-4z}=\frac{1-u}{1+u}.
\end{equation}
Furthermore, $c(z)=\frac{1-\sqrt{1-4z}}{2z}=1+u$.
Therefore
\begin{equation}\label{eq:zc2-u-substitution}
  zc(z)^2 =\frac{u}{(1+u)^2}(1+u)^2 =u.
\end{equation}
By \cref{eq:C-closed-form,eq:sqrt-u-substitution,eq:zc2-u-substitution},
\begin{align*}
  \Ctrans(g)\!\left(\frac{u}{(1+u)^2}\right)=\frac1{\dfrac{1-u}{1+u}\,g(u)}=\frac{1+u}{1-u}\frac1{g(u)}.
\end{align*}
Substitution into \cref{eq:q-moment-before-substitution} gives
\cref{eq:q-moment-general}.
\end{proof}

\begin{definition}\label{def:central-binomial-functional}
Let \(\Afunc:\Ring[X]\to\Ring\) be the \(\Ring\)-linear functional defined by
\begin{align*}
\Afunc(X^n)=\binom{2n}{n}\qquad(n\ge0).
\end{align*}
\end{definition}

Its generating function is
\begin{equation}\label{eq:central-binomial-generating-function}
\sum_{n\ge0}\Afunc(X^n)z^n=\sum_{n\ge0}\binom{2n}{n}z^n=\frac{1}{\sqrt{1-4z}}.
\end{equation}

\begin{proposition}\label{prop:q-orthogonality}
For all \(i,j\ge 0\), we have
\begin{equation}\label{eq:q-orthogonality}
\Afunc(q_iq_j) =
\begin{cases}
 1,&i=j=0,\\
 2,&i=j\ge 1,\\
 0,&i\neq j.
 \end{cases}
\end{equation}
\end{proposition}
\begin{proof}
First apply \(\Afunc\), coefficientwise in \(u\), to
\cref{eq:q-generating-function}. Since
\[
 \frac{1}{(1+u)^2-Xu}
 =
 \frac{1}{(1+u)^2}
 \frac{1}{1-\dfrac{Xu}{(1+u)^2}},
\]
we obtain
\begin{align} \label{eq:A-on-q-generating-function}
 \sum_{n\ge 0}\Afunc(q_n)u^n=
 \frac{1-u^2}{(1+u)^2}
 \sum_{k\ge 0}\Afunc(X^k)
 \left(\frac{u}{(1+u)^2}\right)^k
 =\frac{1-u^2}{(1+u)^2}\frac{1}{\sqrt{1-\dfrac{4u}{(1+u)^2}}}.
\end{align}
By uniqueness of the square root with constant term \(1\),
\begin{equation}\label{eq:formal-square-root-substitution}
 \sqrt{1-\frac{4u}{(1+u)^2}}=\frac{1-u}{1+u}.
\end{equation}
Substituting \cref{eq:formal-square-root-substitution} into
\cref{eq:A-on-q-generating-function} gives $\sum_{n\ge 0}\Afunc(q_n)u^n=1$.
Therefore
\begin{equation}\label{eq:A-q-values}
 \Afunc(q_0)=1,
 \qquad
 \Afunc(q_n)=0\quad(n\ge 1).
\end{equation}

If one of \(i,j\) is zero and the other is positive, then $\Afunc(q_iq_j)=0$ by \cref{eq:A-q-values}. If \(i,j\ge 1\) and \(i\neq j\), then \cref{eq:q-product-off-diagonal} gives $\Afunc(q_iq_j)=\Afunc(q_{i+j})+\Afunc(q_{|i-j|})=0$, because both indices are positive. Finally, if \(i=j\ge 1\), then \cref{eq:q-product-diagonal} gives 
$\Afunc(q_i^2)=\Afunc(q_{2i})+2\Afunc(q_0)=2$.
The case \(i=j=0\) is \(\Afunc(q_0^2)=\Afunc(1)=1\).
\end{proof}

\section{The family $(1+ax)/(1+bx)$}\label{sec:c12}

Throughout this section, $R=\mathbb{Q}[a,b]$.
Let
\[g_{a,b}(x)=\frac{1+ax}{1+bx},\qquad M(z)=\Ctrans(g_{a,b})(z)=\sum_{n\ge0}\mu_nz^n,
\]
and let \(\Lfunc(X^n)=\mu_n\). By \cref{prop:C-closed-form},
\begin{equation}\label{c12:eq:moment-gf}
M(z) =\frac{1+bzc(z)^2}{\sqrt{1-4z}\bigl(1+azc(z)^2\bigr)}.
\end{equation}
Set $\nu_n=\Lfunc(q_n)$ ($n\ge0$). The universal identity \cref{prop:q-moment-general} gives
\begin{equation}\label{c12:eq:nu-gf}
\sum_{n\ge0}\nu_nu^n=\frac{1+bu}{1+au}.
\end{equation}
Consequently,
\begin{equation}\label{c12:eq:nu-explicit}
  \nu_0=1,
  \qquad
  \nu_n=(b-a)(-a)^{n-1}
  \quad(n\ge1).
\end{equation}

For the tridiagonalization, it is convenient to introduce
\begin{equation}\label{c12:eq:rho-delta}
\rho=-a,
\qquad
\Delta=b-a.
\end{equation}
Then \cref{c12:eq:nu-explicit} becomes
\begin{equation}\label{c12:eq:nu-rho}
\nu_0=1,
\qquad
\nu_n=\Delta\rho^{n-1}
\quad(n\ge 1).
\end{equation}
In particular,
\begin{equation}\label{c12:eq:nu-geometric}
\nu_{n+1}=\rho\nu_n
\qquad(n\ge 1).
\end{equation}
We shall repeatedly use the elementary identities
\begin{equation}\label{c12:eq:rho-delta-identities}
\Delta-\rho=b,
\qquad
\Delta-2\rho=a+b,
\qquad
-\rho b=ab.
\end{equation}

\subsection{Tridiagonalization of the Gram matrix}\label{c12:sec:tridiagonalization}

Define a new monic polynomial sequence by
\begin{equation}\label{c12:eq:p-definition}
\begin{aligned}
p_0(X)&=q_0(X)=1,\\
p_n(X)&=q_n(X)-\rho q_{n-1}(X)
       =q_n(X)+a q_{n-1}(X)
       \qquad(n\ge 1).
\end{aligned}
\end{equation}
Because \(q_n\) is monic of degree \(n\), each \(p_n\) is also monic of degree \(n\).
For \(i,j\ge 0\), set
\begin{equation}\label{c12:eq:Gamma-entry}
\Gamma_{i,j}=\Lfunc(p_i p_j).
\end{equation}

\begin{proposition}
\label{c12:prop:tridiagonal}
For every \(i,j\ge 0\),
\begin{equation}\label{c12:eq:Gamma-cases}
\Gamma_{i,j}
=
\begin{cases}
1,
& i=j=0,\\[2mm]
b,
& \{i,j\}=\{0,1\},\\[2mm]
2+ab,
& i=j=1,\\[2mm]
2(1+ab),
& i=j\ge 2,\\[2mm]
a+b,
& |i-j|=1\ \text{and}\ \min\{i,j\}\ge 1,\\[2mm]
0,
& \text{otherwise}.
\end{cases}
\end{equation}
Consequently, for \(n\ge 2\),
\begin{equation}\label{c12:eq:Gamma-matrix}
\Gamma_n
:=
\bigl(\Gamma_{i,j}\bigr)_{0\le i,j\le n}
=
\begin{pmatrix}
1 & b & 0 & \cdots & 0\\
b & 2+ab & a+b & \ddots & \vdots\\
0 & a+b & 2(1+ab) & \ddots & 0\\
\vdots & \ddots & \ddots & \ddots & a+b\\
0 & \cdots & 0 & a+b & 2(1+ab)
\end{pmatrix}.
\end{equation}
\end{proposition}
\begin{proof}
The bilinear form is symmetric because \(\Lfunc(fg)=\Lfunc(gf)\).  It is therefore enough to compute \(\Gamma_{i,j}\) for \(0\le i\le j\). We consider the following cases.

\noindent\emph{Case 1: the entries involving \(p_0\).}
Since \(p_0=q_0=1\), $\Gamma_{0,0}=\Lfunc(q_0)=\nu_0=1$.
Also,
\begin{align*}
\Gamma_{0,1}=\Lfunc(q_1-\rho q_0)=\nu_1-\rho\nu_0=\Delta-\rho=b,
\end{align*}
where the last equality follows from \cref{c12:eq:rho-delta-identities}.
For \(j\ge 2\), the recurrence \(\nu_j=\rho\nu_{j-1}\) is applicable because \(j-1\ge 1\).
Therefore
\begin{align*}
\Gamma_{0,j}=\Lfunc(q_j-\rho q_{j-1})=\nu_j-\rho\nu_{j-1}=0.
\end{align*}

\noindent\emph{Case 2: the entry \(\Gamma_{1,1}\).}
By \cref{eq:q-product-diagonal}, $q_1^2=q_2+2q_0$.
Hence
\begin{align*}
\Gamma_{1,1}&=\Lfunc\bigl((q_1-\rho q_0)^2\bigr)=\Lfunc(q_1^2)-2\rho\Lfunc(q_1)+\rho^2\Lfunc(q_0)=(\nu_2+2)-2\rho\nu_1+\rho^2\nu_0\\
&=\rho\Delta+2-2\rho\Delta+\rho^2=2-\rho\Delta+\rho^2=2-\rho(\Delta-\rho)=2-\rho b=2+ab.
\end{align*}

\noindent\emph{Case 3: the entry \(\Gamma_{1,2}\).}
Using $q_1q_2=q_3+q_1$, $q_1^2=q_2+2q_0$, we find
\begin{align*}
\Gamma_{1,2}&= \Lfunc\bigl((q_1-\rho q_0)(q_2-\rho q_1)\bigr)=\Lfunc(q_1q_2)-\rho\Lfunc(q_1^2)-\rho\Lfunc(q_2)+\rho^2\Lfunc(q_1)
\\&=(\nu_3+\nu_1)-\rho(\nu_2+2)-\rho\nu_2+\rho^2\nu_1
=\rho^2\Delta+\Delta-\rho(\rho\Delta+2)-\rho(\rho\Delta)+\rho^2\Delta
\\&=\Delta-2\rho=a+b.
\end{align*}

\noindent\emph{Case 4: the entries \(\Gamma_{1,j}\) for \(j\ge 3\).}
By \cref{eq:q-product-off-diagonal}, $q_1q_j=q_{j+1}+q_{j-1}$, $q_1q_{j-1}=q_j+q_{j-2}$.
Therefore
\begin{align*}
\Gamma_{1,j}&=\Lfunc\bigl((q_1-\rho q_0)(q_j-\rho q_{j-1})\bigr)=\Lfunc(q_1q_j)-\rho\Lfunc(q_1q_{j-1})-\rho\Lfunc(q_j)+\rho^2\Lfunc(q_{j-1})
\\&=\nu_{j+1}+\nu_{j-1}-\rho(\nu_j+\nu_{j-2})-\rho\nu_j+\rho^2\nu_{j-1}\\
&=(\nu_{j+1}-\rho\nu_j)-\rho(\nu_j-\rho\nu_{j-1})+(\nu_{j-1}-\rho\nu_{j-2}).
\end{align*}
Since \(j\ge 3\), all three uses of
\cref{c12:eq:nu-geometric} involve indices at least \(1\).  Each parenthesis
is therefore zero, and $\Gamma_{1,j}=0$ ($j\ge 3$).

\noindent\emph{Case 5: the diagonal entries \(\Gamma_{j,j}\) for \(j\ge 2\).}
By \cref{eq:q-product-diagonal,eq:q-product-off-diagonal},
\[
q_j^2=q_{2j}+2q_0,
\qquad
q_jq_{j-1}=q_{2j-1}+q_1,
\qquad
q_{j-1}^2=q_{2j-2}+2q_0.
\]
Thus
\begin{align*}
\Gamma_{j,j}
&=\Lfunc\bigl((q_j-\rho q_{j-1})^2\bigr)=\Lfunc(q_j^2)-2\rho\Lfunc(q_jq_{j-1})+\rho^2\Lfunc(q_{j-1}^2)
\\&=\bigl(\nu_{2j}-2\rho\nu_{2j-1}+\rho^2\nu_{2j-2}\bigr)+2-2\rho\nu_1+2\rho^2.
\end{align*}
The first parenthesis vanishes, because
\begin{align*}
\nu_{2j}-2\rho\nu_{2j-1}+\rho^2\nu_{2j-2}=(\nu_{2j}-\rho\nu_{2j-1})-\rho(\nu_{2j-1}-\rho\nu_{2j-2})=0.
\end{align*}
Since \(\nu_1=\Delta\), the remaining terms give
\begin{align*}
\Gamma_{j,j}=2-2\rho\Delta+2\rho^2=2-2\rho b=2(1+ab).
\end{align*}

\noindent\emph{Case 6: the adjacent entries \(\Gamma_{j,j+1}\) for \(j\ge 2\).}
The Chebyshev product formulas give
\[
q_{j+1}q_j=q_{2j+1}+q_1, \quad q_{j+1}q_{j-1}=q_{2j}+q_2,\quad q_j^2=q_{2j}+2q_0,\quad q_jq_{j-1}=q_{2j-1}+q_1.
\]
Therefore, we have
\begin{align*}
\Gamma_{j+1,j}&=\Lfunc\bigl((q_{j+1}-\rho q_j)(q_j-\rho q_{j-1})\bigr)
\\&=\Lfunc(q_{j+1}q_j)-\rho\Lfunc(q_{j+1}q_{j-1})-\rho\Lfunc(q_j^2)+\rho^2\Lfunc(q_jq_{j-1})\\
&=\bigl(\nu_{2j+1}-2\rho\nu_{2j} +\rho^2\nu_{2j-1}\bigr)+\nu_1-\rho\nu_2-2\rho+\rho^2\nu_1.
\end{align*}
The first parenthesis is zero by
\cref{c12:eq:nu-geometric}.  Since
\(\nu_2=\rho\nu_1\), the remaining expression becomes
$\nu_1-\rho\nu_2-2\rho+\rho^2\nu_1=\nu_1-\rho^2\nu_1-2\rho+\rho^2\nu_1=\nu_1-2\rho =a+b$.
Hence $\Gamma_{j,j+1}=a+b$ ($j\ge 2$).

\noindent\emph{Case 7: all entries at distance at least two vanish.}
Let \(i\ge j+2\) and \(j\ge 2\).  Put $s=i+j$, $k=i-j$.
Then \(k\ge 2\).  The product formulas yield
\begin{align*}
q_iq_j=q_s+q_k,\quad q_iq_{j-1}=q_{s-1}+q_{k+1},\quad
q_{i-1}q_j=q_{s-1}+q_{k-1},\quad q_{i-1}q_{j-1}=q_{s-2}+q_k.
\end{align*}
Consequently,
\begin{align*}
\Gamma_{i,j}&=\Lfunc\bigl((q_i-\rho q_{i-1})(q_j-\rho q_{j-1})\bigr)\\
&=(\nu_s+\nu_k)
-\rho(\nu_{s-1}+\nu_{k+1})
-\rho(\nu_{s-1}+\nu_{k-1})
+\rho^2(\nu_{s-2}+\nu_k)\\
&=\bigl(\nu_s-2\rho\nu_{s-1}+\rho^2\nu_{s-2}\bigr)
+\bigl((1+\rho^2)\nu_k-\rho\nu_{k+1}-\rho\nu_{k-1}\bigr)
\\& =(\nu_s-\rho\nu_{s-1}) -\rho(\nu_{s-1}-\rho\nu_{s-2})+(\nu_k-\rho\nu_{k-1})-\rho(\nu_{k+1}-\rho\nu_k)
\\& =0.
\end{align*}

Combining Cases 1--7, and using symmetry, proves \cref{c12:eq:Gamma-cases} and hence \cref{c12:eq:Gamma-matrix}.
\end{proof}

\subsection{Hankel determinants and their generating function}
\label{c12:sec:determinants}

\begin{proof}[Proof of \cref{thm:c12}]
Because the polynomials \(p_n\) are monic of degree \(n\),
\cref{lem:monic-change-basis} and \cref{c12:prop:tridiagonal} imply $h_n=\det\Gamma_n$, ($n\ge 0$).
We first compute the initial values.  For \(n=0\), $h_0=\det(1)=1$.
For \(n=1\),
\begin{align}\label{c12:eq:h1-computation}
h_1=
\det
\begin{pmatrix}
1&b\\
b&2+ab
\end{pmatrix}=2+ab-b^2.
\end{align}

Now let \(n\ge 2\).  The last row of \(\Gamma_n\) contains only two
nonzero entries:
\[
(\Gamma_n)_{n,n-1}=a+b,
\qquad
(\Gamma_n)_{n,n}=2(1+ab).
\]
Expanding \(\det\Gamma_n\) along the last row, the diagonal entry
contributes $2(1+ab)\det\Gamma_{n-1}=2(1+ab)h_{n-1}$.
The entry \(a+b\) in column \(n-1\) has cofactor sign $(-1)^{n+(n-1)}=-1$.
After deleting row \(n\) and column \(n-1\), the resulting minor has, in its last column, a single nonzero entry \(a+b\), located in its last row.  Expanding that minor along its last column therefore gives
$(a+b)\det\Gamma_{n-2}=(a+b)h_{n-2}$.
Thus the total contribution of the \((n,n-1)\)-entry is $-(a+b)^2h_{n-2}$.
We have proved
\begin{equation}\label{c12:eq:det-recurrence}
h_n=2(1+ab)h_{n-1}-(a+b)^2h_{n-2}\qquad(n\ge 2).
\end{equation}

Let $H(x)=\sum_{n\ge 0}h_nx^n$.
Multiplying \cref{c12:eq:det-recurrence} by \(x^n\) and summing over
\(n\ge 2\), we obtain
\begin{align*}
\sum_{n\ge 2}h_nx^n
&=
2(1+ab)\sum_{n\ge 2}h_{n-1}x^n
-(a+b)^2\sum_{n\ge 2}h_{n-2}x^n.
\end{align*}
Each sum can be expressed in terms of \(H(x)\):
\[\sum_{n\ge 2}h_nx^n=H(x)-h_0-h_1x,\quad \sum_{n\ge 2}h_{n-1}x^n=x\bigl(H(x)-h_0\bigr),\quad \sum_{n\ge 2}h_{n-2}x^n=x^2H(x)
\]
Therefore
\begin{align}
\bigl(1-2(1+ab)x+(a+b)^2x^2\bigr)H(x)
&=
h_0+\bigl(h_1-2(1+ab)h_0\bigr)x.
\label{c12:eq:H-before-numerator}
\end{align}
Using \(h_0=1\) and \(h_1=2+ab-b^2\), substitution into \cref{c12:eq:H-before-numerator} gives
\[H(x)=\frac{1-b(a+b)x}{1-2(1+ab)x+(a+b)^2x^2}.\]
This proves \cref{thm:c12}.
\end{proof}

\begin{example}
For \((a,b)=(-2,1)\), the theorem gives $H(x)=\frac{1+x}{(1+x)^2}=\frac{1}{1+x}$.
For \((a,b)=(1,2)\), it gives $H(x)=\frac{1-6x}{1-6x+9x^2}=\frac{1-6x}{(1-3x)^2}$.
These are the special cases recorded in \cite{Barry2020}.
\end{example}

\section{The family $(1+ax)/(1-bx^2)$}\label{sec:c17}

All calculations below are performed over the polynomial ring $\Ring=\mathbb{Q}[a,b]$.
The resulting identities are polynomial identities and may be specialized to arbitrary values of \(a,b\) in every field of
characteristic zero.

Let
\[g_{a,b}(x)=\frac{1+ax}{1-bx^2},\qquad M(z)=\Ctrans(g_{a,b})(z)=\sum_{n\ge0}\mu_nz^n.
\]
By \cref{prop:C-closed-form},
\begin{equation}\label{c17:eq:moment-generating-function}
  M(z)
  =
  \frac{1-bz^2c(z)^4}
       {\sqrt{1-4z}\bigl(1+azc(z)^2\bigr)}.
\end{equation}
Let \(\Lfunc(X^n)=\mu_n\), and set $\nu_n=\Lfunc(q_n)$ ($n\ge0$).
For \(g_{a,b}(u)=(1+au)/(1-bu^2)\), \cref{prop:q-moment-general} gives
$\sum_{n\ge0}\nu_nu^n=\frac{1-bu^2}{1+au}$.
Introduce $\rho=-a$. Then
\[\frac{1-bu^2}{1+au}=(1-bu^2)\sum_{n\ge0}\rho^nu^n.
\]
Comparing coefficients gives
\begin{equation}\label{c17:eq:nu-explicit}
  \nu_0=1,
  \qquad
  \nu_1=\rho,
  \qquad
  \nu_n=(\rho^2-b)\rho^{n-2}
  \quad(n\ge2).
\end{equation}
In particular,
\begin{equation}\label{c17:eq:nu-eventual-recurrence}
  \nu_{n+1}=\rho\nu_n
  \qquad(n\ge2).
\end{equation}
It is useful to record the one-step defects $\varepsilon_n=\nu_n-\rho\nu_{n-1}$, ($n\ge1$).
\cref{c17:eq:nu-explicit,c17:eq:nu-eventual-recurrence} yield
\begin{equation}\label{c17:eq:epsilon-values}
  \varepsilon_1=0,
  \qquad
  \varepsilon_2=-b,
  \qquad
  \varepsilon_n=0
  \quad(n\ge3).
\end{equation}

\subsection{Pentadiagonalization of the Gram matrix} \label{c17:sec:pentadiagonalization}

Define a new monic polynomial basis by
\begin{equation}\label{c17:eq:p-definition}
\begin{aligned}
  p_0(X)&=q_0(X)=1,\\
  p_n(X)&=q_n(X)-\rho q_{n-1}(X)
          =q_n(X)+a q_{n-1}(X),
          \qquad n\ge1.
\end{aligned}
\end{equation}
Since \(q_n\) is monic of degree \(n\), each \(p_n\) is also monic of degree \(n\).
Set $\sigma=\rho(b-1)=a(1-b)$
and $\Gamma_{i,j}=\Lfunc(p_ip_j)$.

\begin{proposition}[Pentadiagonal Gram matrix]
\label{c17:prop:pentadiagonal-Gram}
For all \(i,j\ge0\),
\begin{equation}\label{c17:eq:Gamma-cases}
\Gamma_{i,j}=
\begin{cases}
  1,
  &i=j=0,\\[1mm]
  0,
  &\{i,j\}=\{0,1\},\\[1mm]
  -b,
  &\{i,j\}=\{0,2\},\\[1mm]
  2-b,
  &i=j=1,\\[1mm]
  2,
  &i=j\ge2,\\[1mm]
  \sigma,
  &|i-j|=1\text{ and }\min\{i,j\}\ge1,\\[1mm]
  -b,
  &|i-j|=2,\\[1mm]
  0,
  &\text{otherwise}.
\end{cases}
\end{equation}
Consequently, for \(n\ge2\),
\begin{equation}\label{c17:eq:Gamma-matrix}
\Gamma_n=(\Gamma_{i,j})_{0\le i,j\le n}
=
\begin{pmatrix}
1&0&-b&0&0&\cdots&0\\
0&2-b&\sigma&-b&0&\ddots&\vdots\\
-b&\sigma&2&\sigma&-b&\ddots&0\\
0&-b&\sigma&2&\sigma&\ddots&0\\
0&0&-b&\sigma&2&\ddots&-b\\
\vdots&\ddots&\ddots&\ddots&\ddots&\ddots&\sigma\\
0&\cdots&0&0&-b&\sigma&2
\end{pmatrix}.
\end{equation}
\end{proposition}
\begin{proof}
The bilinear form is symmetric, because \(\Lfunc(fg)=\Lfunc(gf)\).  It is therefore enough to compute \(\Gamma_{i,j}\) for \(0\le i\le j\). We consider the following cases.

\noindent\emph{Case 1: the entries in row and column zero.}
Since \(p_0=q_0=1\), $\Gamma_{0,0}=\Lfunc(q_0)=\nu_0=1$.
Moreover,
\begin{align*}
\Gamma_{0,1}=\Lfunc(q_1-\rho q_0)=\rho-\rho=0, \text{ and }\Gamma_{0,2}=\Lfunc(q_2-\rho q_1)=\nu_2-\rho\nu_1=-b.
\end{align*}
For \(j\ge3\), \cref{c17:eq:nu-eventual-recurrence} gives $\Gamma_{0,j}=\nu_j-\rho\nu_{j-1}=0$.

\noindent\emph{Case 2: the entry \(\Gamma_{1,1}\).}
By \cref{eq:q-product-diagonal}, $q_1^2=q_2+2q_0$. Therefore
\begin{align*}
\Gamma_{1,1}=\Lfunc\bigl((q_1-\rho q_0)^2\bigr)=\Lfunc(q_1^2)-2\rho\Lfunc(q_1)+\rho^2\Lfunc(q_0)
=(\nu_2+2)-2\rho\nu_1+\rho^2\nu_0=2-b.
\end{align*}

\noindent\emph{Case 3: the entry \(\Gamma_{1,2}\).}
The product identities give $q_1q_2=q_3+q_1$, $q_1^2=q_2+2q_0$.
Thus
\begin{align*}
\Gamma_{1,2}&=\Lfunc\bigl((q_1-\rho q_0)(q_2-\rho q_1)\bigr)
  =\Lfunc(q_1q_2)-\rho\Lfunc(q_1^2) -\rho\Lfunc(q_2) +\rho^2\Lfunc(q_1)\\
  &=(\nu_3+\nu_1)-\rho(\nu_2+2)-\rho\nu_2+\rho^2\nu_1=\rho\nu_2+\rho-\rho\nu_2-2\rho-\rho\nu_2+\rho^3\\
  &=\rho(b-1)=\sigma.
\end{align*}

\noindent\emph{Case 4: the entry \(\Gamma_{1,3}\).}
We have $q_1q_3=q_4+q_2$ and $q_1q_2=q_3+q_1$.
Consequently,
\begin{align*}
\Gamma_{1,3}&=\Lfunc\bigl((q_1-\rho q_0)(q_3-\rho q_2)\bigr)=\nu_4+\nu_2 -\rho(\nu_3+\nu_1)-\rho\nu_3+\rho^2\nu_2\\
  &=(\nu_4-\rho\nu_3) +(\nu_2-\rho\nu_1) -\rho(\nu_3-\rho\nu_2)=-b.
\end{align*}

\noindent\emph{Case 5: the entries \(\Gamma_{1,j}\) for \(j\ge4\).}
Using $q_1q_j=q_{j+1}+q_{j-1}$ and $q_1q_{j-1}=q_j+q_{j-2}$,
we obtain
\begin{align*}
\Gamma_{1,j} &=\nu_{j+1}+\nu_{j-1}-\rho(\nu_j+\nu_{j-2})-\rho\nu_j+\rho^2\nu_{j-1}\\
&=(\nu_{j+1}-\rho\nu_j) -\rho(\nu_j-\rho\nu_{j-1}) +(\nu_{j-1}-\rho\nu_{j-2}).
\end{align*}
For \(j\ge4\), all three parentheses vanish by \cref{c17:eq:nu-eventual-recurrence}.  Hence $\Gamma_{1,j}=0$ ($j\ge4$).

\noindent\emph{Case 6: diagonal entries \(\Gamma_{j,j}\) for \(j\ge2\).}
By \cref{eq:q-product-diagonal,eq:q-product-off-diagonal},
\[
  q_j^2=q_{2j}+2q_0,
  \qquad
  q_jq_{j-1}=q_{2j-1}+q_1,
  \qquad
  q_{j-1}^2=q_{2j-2}+2q_0.
\]
Therefore
\begin{align*}
\Gamma_{j,j}&=\Lfunc\bigl((q_j-\rho q_{j-1})^2\bigr)
=(\nu_{2j}+2) -2\rho(\nu_{2j-1}+\nu_1) +\rho^2(\nu_{2j-2}+2)\\
&=\bigl(\nu_{2j}-2\rho\nu_{2j-1} +\rho^2\nu_{2j-2}\bigr)+2-2\rho\nu_1+2\rho^2
\\ &= (\nu_{2j}-\rho\nu_{2j-1})-\rho(\nu_{2j-1}-\rho\nu_{2j-2})+2-2\rho^2+2\rho^2
\\ &= 2.
\end{align*}

\noindent\emph{Case 7: adjacent entries \(\Gamma_{j+1,j}\) for \(j\ge2\).}
The product identities give
\[q_{j+1}q_j=q_{2j+1}+q_1,\quad q_{j+1}q_{j-1}=q_{2j}+q_2,\quad q_j^2=q_{2j}+2q_0,\quad q_jq_{j-1}=q_{2j-1}+q_1.
\]
Hence
\begin{align*}
 \Gamma_{j+1,j}&=\Lfunc\bigl((q_{j+1}-\rho q_j)(q_j-\rho q_{j-1})\bigr)\\
  &=(\nu_{2j+1}+\nu_1) -\rho(\nu_{2j}+\nu_2) -\rho(\nu_{2j}+2) +\rho^2(\nu_{2j-1}+\nu_1)\\
  &=\bigl(\nu_{2j+1}-2\rho\nu_{2j} +\rho^2\nu_{2j-1}\bigr) +\nu_1-\rho\nu_2-2\rho+\rho^2\nu_1.
\end{align*}
The first parenthesis is zero.  Using
\(\nu_1=\rho\) and \(\nu_2=\rho^2-b\), we obtain $\Gamma_{j+1,j}=\sigma$  ($j\ge2$).

\noindent\emph{Case 8: entries at distance at least two, with both indices at least two.}
Let \(i\ge j+2\) and \(j\ge2\).  Put $m=i+j$, $k=i-j\ge2$.
The product identities give
\begin{align*}
q_iq_j=q_m+q_k,\quad q_iq_{j-1}=q_{m-1}+q_{k+1}, \quad q_{i-1}q_j=q_{m-1}+q_{k-1},\quad q_{i-1}q_{j-1}=q_{m-2}+q_k.
\end{align*}
Therefore
\begin{align*}
\Gamma_{i,j}&=(\nu_m+\nu_k)-\rho(\nu_{m-1}+\nu_{k+1})-\rho(\nu_{m-1}+\nu_{k-1})+\rho^2(\nu_{m-2}+\nu_k)
  \\&=\bigl(\nu_m-2\rho\nu_{m-1}+\rho^2\nu_{m-2}\bigr)+\bigl((1+\rho^2)\nu_k-\rho\nu_{k+1}-\rho\nu_{k-1}\bigr)
  \\&= (\nu_m-\rho\nu_{m-1})-\rho(\nu_{m-1}-\rho\nu_{m-2})+(\nu_k-\rho\nu_{k-1})-\rho(\nu_{k+1}-\rho\nu_k)
  \\&= 0+\varepsilon_k-\rho\varepsilon_{k+1}.
\end{align*}
If \(k=2\), then \cref{c17:eq:epsilon-values} gives $B_2=\varepsilon_2-\rho\varepsilon_3=-b$.
If \(k\ge3\), then both defects vanish and \(B_k=0\).  Thus
\[\Gamma_{i,j}=
\begin{cases}
-b,&i-j=2,\\
0,&i-j\ge3.
  \end{cases}
\]

Combining Steps 1--8 and using symmetry proves
\cref{c17:eq:Gamma-cases}, hence also the matrix form
\cref{c17:eq:Gamma-matrix}.
\end{proof}

Since the basis \((p_n)\) is monic, \cref{lem:monic-change-basis} gives
\begin{equation}\label{c17:eq:hankel-equals-Gamma}
  h_n=\det\Gamma_n
  \qquad(n\ge0).
\end{equation}

\subsection{A recurrence for a pentadiagonal matrix}\label{c17:sec:continuant}

We now derive the determinant recurrence needed for \cref{c17:eq:Gamma-matrix}.  The argument is stated over an arbitrary commutative ring.

\begin{lemma}\label{c17:lem:pentadiagonal-continuant}
Let \((A_m)_{m\ge0}\) be the leading principal sections of an infinite
symmetric matrix, where \(A_0\) is the empty matrix and
\(D_m=\det A_m\), with \(D_0=1\).  Suppose that, for every \(m\ge3\),
the last row of \(A_m\) has the form
\[(0,\ldots,0,e,c,d),
\]
so that the matrix has bandwidth two and the newly appended diagonal,
first off-diagonal, and second off-diagonal entries are the constants
\(d,c,e\), respectively.

Then, for every \(m\ge5\),
\begin{equation}\label{c17:eq:fifth-order-continuant}
D_m=(d-e)D_{m-1}+(de-c^2)D_{m-2}+e(c^2-de)D_{m-3}+e^3(e-d)D_{m-4}+e^5D_{m-5}.
\end{equation}
Define
\begin{align}
A=d-2e,\quad B=c^2-2de+2e^2,\quad C=e^2(d-2e),\quad E=e^4 \label{c17:eq:ABCE-cont}
\end{align}
and
\begin{equation}\label{c17:eq:R-definition}
  R_m=D_m-AD_{m-1}+BD_{m-2}-CD_{m-3}+ED_{m-4}.
\end{equation}
Then
\begin{equation}\label{c17:eq:R-geometric}
  R_m=eR_{m-1}
  \qquad(m\ge5).
\end{equation}
In particular, if \(R_4=0\), then
\begin{equation}\label{c17:eq:fourth-order-continuant}
  D_m=AD_{m-1}-BD_{m-2}+CD_{m-3}-ED_{m-4}
  \qquad(m\ge4).
\end{equation}
\end{lemma}

\begin{proof}
We use one-based indices inside this proof.  Let \(U_m\) be the cofactor
of the entry in position \((m,m-1)\) of \(A_m\), and let \(V_m\) be the
cofactor of the entry in position \((m,m-2)\).  Expansion along the last
row gives
\begin{equation}\label{c17:eq:D-expand-last-row}
  D_m=dD_{m-1}+cU_m+eV_m.
\end{equation}
We next compute \(U_m\), keeping every cofactor sign explicit.
Let $B_m=A_m^{(m\mid m-1)}$ denote the matrix obtained by deleting row \(m\) and column \(m-1\).
Since $(-1)^{m+(m-1)}=-1$, we have
\begin{equation}\label{c17:eq:Um-as-Bm}
  U_m=-\det B_m.
\end{equation}
The last column of \(B_m\), inherited from column \(m\) of \(A_m\), has the entry \(e\) in row position \(m-2\), the entry \(c\) in row position \(m-1\), and zeros elsewhere.  Laplace expansion along this last column gives
\begin{equation}\label{c17:eq:det-Bm-expansion}
\det B_m=(-1)^{(m-2)+(m-1)}eK_m+(-1)^{(m-1)+(m-1)}cD_{m-2}=-eK_m+cD_{m-2},
\end{equation}
where \(K_m\) is the determinant obtained from \(A_{m-1}\) by deleting
row \(m-2\) and column \(m-1\).  Clearly, we have $K_m=-U_{m-1}$.
Substituting this into \cref{c17:eq:det-Bm-expansion} and then using \cref{c17:eq:Um-as-Bm} yields
\begin{equation}\label{c17:eq:U-recurrence}
U_m=-eU_{m-1}-cD_{m-2}.
\end{equation}

Now compute \(V_m\).  Let $C_m=A_m^{(m\mid m-2)}$.
Since \((-1)^{m+(m-2)}=1\), we have \(V_m=\det C_m\).  The last column of \(C_m\) again comes from column \(m\) of \(A_m\) and has the two nonzero entries \(e\) and \(c\).  Expanding along that column gives
$$V_m=-eW_{m-1}+cL_m,$$
where \(W_{m-1}\) is the determinant obtained from \(A_{m-1}\) by
deleting its penultimate row and penultimate column, and \(L_m\) is the
minor determinant obtained from \(A_{m-1}\) by deleting row \(m-1\)
and column \(m-2\).  Since $U_{m-1}=(-1)^{(m-1)+(m-2)}L_m=-L_m$, we have \(L_m=-U_{m-1}\).  Therefore
\begin{equation}\label{c17:eq:V-recurrence}
  V_m=-cU_{m-1}-eW_{m-1}.
\end{equation}

After the penultimate row and column have been deleted from
\(A_{m-1}\), its final index is coupled to the preceding leading block
only through the second off-diagonal entry \(e\).  Therefore the
remaining matrix has the bordered form
\[
  \begin{pmatrix}
    A_{m-3}&e\mathbf e_{m-3}\\
    e\mathbf e_{m-3}^{\mathsf T}&d
  \end{pmatrix},
\]
where \(\mathbf e_{m-3}\) is the final coordinate vector.  Expanding
along the last row, the diagonal entry contributes \(dD_{m-3}\).  The
only other nonzero entry in that row is \(e\), whose cofactor has sign
\(-1\); the remaining minor has a final column containing only the
entry \(e\), and its expansion contributes \(eD_{m-4}\).  Hence the
off-diagonal contribution is \(-e\cdot eD_{m-4}=-e^2D_{m-4}\), and
therefore
\begin{equation}\label{c17:eq:W-formula}
  W_{m-1}=dD_{m-3}-e^2D_{m-4}.
\end{equation}
Substituting \cref{c17:eq:U-recurrence,c17:eq:V-recurrence,c17:eq:W-formula} into
\cref{c17:eq:D-expand-last-row} gives
\begin{align}
  D_m
  &=dD_{m-1}
    +c(-eU_{m-1}-cD_{m-2})
    +e(-cU_{m-1}-eW_{m-1})\nonumber\\
  &=dD_{m-1}-c^2D_{m-2}-2ceU_{m-1}
    -de^2D_{m-3}+e^4D_{m-4}.
\label{c17:eq:D-with-U}
\end{align}
Apply the same identity with \(m\) replaced by \(m-1\):
$$D_{m-1}=dD_{m-2}-c^2D_{m-3}-2ceU_{m-2}-de^2D_{m-4}+e^4D_{m-5}.$$
Rearranging, without dividing by any parameter, gives
\begin{equation}\label{c17:eq:2ceU}
  2ceU_{m-2}
  =dD_{m-2}-c^2D_{m-3}-de^2D_{m-4}
   +e^4D_{m-5}-D_{m-1}.
\end{equation}
On the other hand, \cref{c17:eq:U-recurrence} implies
\begin{align}
-2ceU_{m-1}=2ce^2U_{m-2}+2c^2eD_{m-3}.
\label{c17:eq:minus-2ceU}
\end{align}
Multiply \cref{c17:eq:2ceU} by \(e\), substitute the result into
\cref{c17:eq:minus-2ceU}, and then substitute into
\cref{c17:eq:D-with-U}.  This gives
\begin{align*}
D_m&=dD_{m-1}-c^2D_{m-2}-de^2D_{m-3}+e^4D_{m-4}
\\&\quad+e\bigl(dD_{m-2}-c^2D_{m-3}-de^2D_{m-4}+e^4D_{m-5}-D_{m-1}\bigr)+2c^2eD_{m-3}
\\&= (d-e)D_{m-1}+(de-c^2)D_{m-2}+e(c^2-de)D_{m-3}+e^3(e-d)D_{m-4}+e^5D_{m-5}.
\end{align*}
which is \cref{c17:eq:fifth-order-continuant}.

It remains to verify \cref{c17:eq:R-geometric}.  From
\cref{c17:eq:ABCE-cont},
\begin{align*}
A+e=d-e,\quad -(B+eA)&=de-c^2,\quad C+eB=e(c^2-de),\quad -(E+eC)=e^3(e-d),
\end{align*}
and $eE=e^5$.
Thus the identity $R_m=eR_{m-1}$ expands exactly to \cref{c17:eq:fifth-order-continuant}.  If \(R_4=0\),
then \cref{c17:eq:R-geometric} implies inductively that \(R_m=0\) for every \(m\ge4\), which is precisely \cref{c17:eq:fourth-order-continuant}.
\end{proof}

\subsection{The Hankel generating function}\label{c17:sec:application-continuant}

Let \(A_m\) be the \(m\times m\) leading principal section of the
infinite matrix in \cref{c17:eq:Gamma-matrix}, with \(A_0\) empty, and set $D_m=\det A_m$, $D_0=1$.
By \cref{c17:eq:hankel-equals-Gamma}, we have $h_n=D_{n+1}$.
For every \(m\ge3\), the final diagonal, first off-diagonal, and second
off-diagonal entries are $d=2$, $c=\sigma$, $e=-b$.
Therefore the constants in \cref{c17:eq:ABCE-cont} are
\begin{align}
& A=d-2e=2(1+b),\quad  B=c^2-2de+2e^2=\sigma^2+4b+2b^2,\nonumber
\\ &C=e^2(d-2e)=2b^2(1+b), \quad E=e^4=b^4.\label{c17:eq:ABCE-special}
\end{align}

We now calculate \(D_0,D_1,D_2,D_3,D_4\) explicitly.
The first two nonempty sections give
\begin{equation}\label{c17:eq:D0-D2}
  D_0=1,\qquad D_1=1,\qquad
  D_2=\det\begin{pmatrix}1&0\\0&2-b\end{pmatrix}=2-b.
\end{equation}
For \(D_3\), we have
\begin{align}
D_3=\det
\begin{pmatrix}
1&0&-b\\
0&2-b&\sigma\\
-b&\sigma&2
\end{pmatrix}
=4-2b-\sigma^2-2b^2+b^3.
\label{c17:eq:D3}
\end{align}
For \(D_4\), we obtain
\begin{align}
D_4=\det
\begin{pmatrix}
1&0&-b&0\\
0&2-b&\sigma&-b\\
-b&\sigma&2&\sigma\\
0&-b&\sigma&2
\end{pmatrix}
=8-4b-6b^2+2b^3+b^4-(b+4)\sigma^2.
\label{c17:eq:D4}
\end{align}

To use \cref{c17:lem:pentadiagonal-continuant}, we must verify \(R_4=0\).
By \cref{c17:eq:R-definition,c17:eq:ABCE-special},
$$R_4=D_4-2(1+b)D_3+(\sigma^2+4b+2b^2)D_2-2b^2(1+b)D_1+b^4D_0.$$
Using \cref{c17:eq:D0-D2,c17:eq:D3,c17:eq:D4}, we obtain $R_4=0$.

It follows from \cref{c17:lem:pentadiagonal-continuant} that, for every \(m\ge4\),
$$D_m=2(1+b)D_{m-1}-(\sigma^2+4b+2b^2)D_{m-2}+2b^2(1+b)D_{m-3}-b^4D_{m-4}.$$
Using \(h_n=D_{n+1}\), we obtain, for every \(n\ge4\),
\begin{equation}\label{c17:eq:h-fourth-recurrence}
h_n=2(1+b)h_{n-1}-(\sigma^2+4b+2b^2)h_{n-2}+2b^2(1+b)h_{n-3}-b^4h_{n-4}.
\end{equation}
The initial values in \cref{c17:eq:h0123-intro}
follow immediately from \(h_n=D_{n+1}\) and
\cref{c17:eq:D0-D2,c17:eq:D3,c17:eq:D4}.

\begin{proof}[Proof of \cref{thm:c17}]
Let $H(x)=\sum_{n\ge0}h_nx^n$. Set
$$Q(x)=1-2(1+b)x+(\sigma^2+4b+2b^2)x^2-2b^2(1+b)x^3+b^4x^4.$$
By \cref{c17:eq:h-fourth-recurrence}, every coefficient of \(x^n\) in
\(Q(x)H(x)\) is zero for \(n\ge4\).  We now calculate the four
remaining coefficients one by one.

The constant coefficient is $h_0=1$.
The coefficient of \(x\) is $h_1-2(1+b)h_0=(2-b)-2(1+b)=-3b$.
The coefficient of \(x^2\) is
\begin{align*}
h_2-2(1+b)h_1+(\sigma^2+4b+2b^2)h_0=b^2(2+b).
\end{align*}
The coefficient of \(x^3\) is
\begin{align*}
h_3-2(1+b)h_2+(\sigma^2+4b+2b^2)h_1-2b^2(1+b)h_0=-b^4
\end{align*}
We have therefore proved
$$Q(x)H(x)=1-3bx+b^2(2+b)x^2-b^4x^3.$$
Since \(Q(0)=1\), the series \(Q(x)\) is invertible in
\(\Ring[[x]]\), and
\begin{equation}\label{c17:eq:H-s-form}
H(x) =\frac{1-3bx+b^2(2+b)x^2-b^4x^3} {1-2(1+b)x+(\sigma^2+4b+2b^2)x^2-2b^2(1+b)x^3+b^4x^4}.
\end{equation}
By $\sigma^2=a^2(1-b)^2=a^2-2a^2b+a^2b^2$, substitution into \cref{c17:eq:H-s-form} gives exactly
\cref{c17:eq:main-H}, completing the proof of \cref{thm:c17}.
\end{proof}

\begin{example}
When \(b=0\), the formula reduces to $H(x)=\frac1{1-2x+a^2x^2}$.
For example, \(a=2\) gives $H(x)=\frac1{1-2x+4x^2}$, in agreement with \cite[Example 18]{Barry2020}.
\end{example}

\section{The family $(1+ax)/(1-x^3)$}\label{sec:c24}

Throughout this section, $R=\mathbb{Q}[a]$.
As printed, the sign of the linear coefficient in the numerator is incorrect in \cite[Conjecture~24]{Barry2020}.
For the family \cref{c24:eq:ga-intro}, define
$M_a(z):=\Ctrans(g_a)(z)=\sum_{n\ge0}\mu_nz^n$.
By \cref{prop:C-closed-form},
$$M_a(z)=\frac{1-z^3c(z)^6}{\sqrt{1-4z}\bigl(1+azc(z)^2\bigr)}.$$
For the remainder of this section, let $\Lfunc:R[X]\longrightarrow R$, $\Lfunc(X^n)=\mu_n$,
be the functional associated with \(M_a(z)\).

\subsection{An adapted monic basis and its Gram matrix}\label{c24:sec:adapted-basis}

Define $\phi(u)=\frac{u}{(1+u)^2}$, $A_a(u)=\frac{1+au}{1+u}$.
We introduce polynomials \(r_n(X)\in\Ring[X]\) by the generating function
\begin{equation}\label{c24:eq:r-basis-gf}
\mathcal{R}_a(X,u):=\sum_{n\ge0}r_n(X)u^n=\frac{(1+au)(1+u)}{(1+u)^2-Xu}=\frac{A_a(u)}{1-X\phi(u)}.
\end{equation}

\begin{lemma}\label{c24:lem:r-monic}
For every \(n\ge0\), the polynomial \(r_n(X)\) is monic of degree \(n\).
\end{lemma}
\begin{proof}
Using the geometric-series expansion in the variable \(X\phi(u)\),
we have
\begin{align*}
\mathcal{R}_a(X,u)&=\frac{1+au}{1+u}\sum_{k\ge0}X^k\left(\frac{u}{(1+u)^2}\right)^k=\sum_{k\ge0}X^ku^k\frac{1+au}{(1+u)^{2k+1}}.
\end{align*}
To obtain the coefficient of \(u^n\), only indices \(k\le n\) can
contribute.  The term with \(k=n\) contributes
\[
  X^n[u^0]\frac{1+au}{(1+u)^{2n+1}}=X^n,
\]
because the last factor has constant term \(1\).  Every term with
\(k<n\) has degree at most \(k<n\) in \(X\).  Thus the coefficient of
\(u^n\) is a polynomial of degree exactly \(n\), with leading
coefficient \(1\).
\end{proof}

\begin{lemma}\label{c24:lem:resolvent}
Let \(M(z)=\sum_{n\ge0}\mu_nz^n\), and let \(\Lfunc\) be \cref{eq:moment-functional}.  For indeterminates
\(s,t\) with zero constant terms, we have 
\begin{align*}
  \Lfunc\left(\frac{1}{(1-Xs)(1-Xt)}\right)=\frac{sM(s)-tM(t)}{s-t}.
\end{align*}
The quotient on the right is understood as the formal divided
difference
\begin{equation}\label{c24:eq:divided-difference-definition}
  \frac{sM(s)-tM(t)}{s-t}
  :=
  \sum_{k\ge0}\mu_k\sum_{j=0}^ks^{k-j}t^j.
\end{equation}
\end{lemma}
\begin{proof}
Expanding both geometric series gives
\begin{align*}
\Lfunc\left( \frac{1}{(1-Xs)(1-Xt)}\right)
 =\Lfunc\left(\sum_{m,n\ge0}X^{m+n}s^mt^n\right)
 =\sum_{m,n\ge0}\mu_{m+n}s^mt^n.
\end{align*}
Reindexing by \(k=m+n\),
\[\sum_{m,n\ge0}\mu_{m+n}s^mt^n=\sum_{k\ge0}\mu_k\sum_{j=0}^ks^{k-j}t^j.
\]
For every \(k\ge0\),
\[
  \sum_{j=0}^ks^{k-j}t^j
  =\frac{s^{k+1}-t^{k+1}}{s-t}.
\]
Hence the last double sum is exactly the divided difference in
\cref{c24:eq:divided-difference-definition}.
\end{proof}

\begin{proposition}\label{c24:prop:gram-kernel}
Let $K_a(u,v)=\sum_{i,j\ge0}\Lfunc(r_i r_j)u^iv^j$.
Then
\begin{equation}\label{c24:eq:gram-kernel-final}
K_a(u,v)=\frac{\Pi_a(u,v)}{1-uv},
\end{equation}
where
\begin{equation}\label{c24:eq:Pi-def}
\Pi_a(u,v)=1+u+v+u^2+v^2+2uv+(a+1)(u^2v+uv^2)+au^2v^2.
\end{equation}
\end{proposition}
\begin{proof}
By \cref{c24:eq:r-basis-gf,eq:moment-functional},
\begin{align*}
K_a(u,v)=\Lfunc\bigl(\mathcal{R}_a(X,u)\mathcal{R}_a(X,v)\bigr)
=A_a(u)A_a(v)\Lfunc\left( \frac{1}{(1-X\phi(u))(1-X\phi(v))}\right).
\end{align*}
Applying \cref{c24:lem:resolvent} with \(s=\phi(u)\), \(t=\phi(v)\), we obtain
\begin{equation}\label{c24:eq:gram-kernel-divided}
K_a(u,v)=A_a(u)A_a(v)\frac{\phi(u)M_a(\phi(u))-\phi(v)M_a(\phi(v))}{\phi(u)-\phi(v)}.
\end{equation}
We now evaluate \(M_a(\phi(u))\).  Since
\[1-4\phi(u)=1-\frac{4u}{(1+u)^2}=\frac{(1-u)^2}{(1+u)^2},\]
and both formal square roots have constant term \(1\), $\sqrt{1-4\phi(u)}=\frac{1-u}{1+u}$.
Moreover, $c(\phi(u))=\frac{1-\sqrt{1-4\phi(u)}}{2\phi(u)} =1+u$.
Therefore, we have $\phi(u)c(\phi(u))^2=u$.

Using \cref{eq:C-closed-form} and
\(g_a(u)=(1+au)/(1-u^3)\), we find
\begin{align*}
M_a(\phi(u))=\frac{1}{\sqrt{1-4\phi(u)}\,g_a(u)}
=\frac{1+u}{1-u}\frac{1-u^3}{1+au}
=\frac{(1+u)(1+u+u^2)}{1+au}.
\end{align*}
Consequently,
\begin{equation}\label{c24:eq:A-M-cancel}
  A_a(u)M_a(\phi(u))=1+u+u^2.
\end{equation}

For brevity, set $P(w)=1+w+w^2$. Substitution of \cref{c24:eq:A-M-cancel} into
\cref{c24:eq:gram-kernel-divided} gives
$$K_a(u,v)=\frac{A_a(v)\phi(u)P(u)-A_a(u)\phi(v)P(v)}{\phi(u)-\phi(v)}.$$
After clearing the denominators \((1+u)^2(1+v)^2\), this becomes
\begin{equation}\label{c24:eq:kernel-cleared}
K_a(u,v)=\frac{u(1+av)(1+v)P(u)-v(1+au)(1+u)P(v)}{u(1+v)^2-v(1+u)^2}.
\end{equation}
The denominator factors as
\begin{align}
u(1+v)^2-v(1+u)^2=u+2uv+uv^2-v-2uv-u^2v=(u-v)(1-uv).
\label{c24:eq:kernel-den-factor}
\end{align}
We next factor the numerator.  Since $(1+av)(1+v)=1+(a+1)v+av^2$, and analogously with \(u\) in place of \(v\), the numerator in \cref{c24:eq:kernel-cleared} is
\begin{align}
uP(u)-vP(v)+(a+1)uv\bigl(P(u)-P(v)\bigr)+auv\bigl(vP(u)-uP(v)\bigr).
\label{c24:eq:kernel-num-decompose}
\end{align}
The three differences in \cref{c24:eq:kernel-num-decompose} factor as
\begin{align}
uP(u)-vP(v)&=(u-v)\bigl(1+u+v+u^2+uv+v^2\bigr),\label{c24:eq:diff-1}\\
P(u)-P(v)&=(u-v)(1+u+v),\label{c24:eq:diff-2}\\
vP(u)-uP(v)&=(u-v)(uv-1).\label{c24:eq:diff-3}
\end{align}
Substituting \cref{c24:eq:diff-1,c24:eq:diff-2,c24:eq:diff-3} into \cref{c24:eq:kernel-num-decompose}, we obtain
\begin{align*}
(u-v)\Bigl(1+u+v+u^2+uv+v^2+(a+1)uv(1+u+v)+auv(uv-1)\Bigr).
\end{align*}
The coefficient of \(uv\) inside the parentheses is $1+(a+1)-a=2$.
Thus the expression inside the parentheses is precisely
\(\Pi_a(u,v)\) from \cref{c24:eq:Pi-def}.  Cancelling the common factor
\(u-v\) with \cref{c24:eq:kernel-den-factor} proves
\cref{c24:eq:gram-kernel-final}.
\end{proof}

\begin{corollary}\label{c24:cor:gram-matrix}
Let $\mathsf{G}_N=\bigl(\Lfunc(r_ir_j)\bigr)_{0\le i,j\le N-1}$, ($N\ge1$).
Then
\begin{equation}\label{c24:eq:h-det-G}
h_{N-1}=\det\mathsf{G}_N.
\end{equation}
For \(N\ge5\), the matrix is the truncation of
\begin{equation}\label{c24:eq:G-matrix-display}
\mathsf{G}_N=
\begin{pmatrix}
1&1&1&0&0&\cdots\\
1&3&a+2&1&0&\ddots\\
1&a+2&a+3&a+2&1&\ddots\\
0&1&a+2&a+3&a+2&\ddots\\
0&0&1&a+2&a+3&\ddots\\
\vdots&\ddots&\ddots&\ddots&\ddots&\ddots
\end{pmatrix}.
\end{equation}
More explicitly,
\begin{equation}\label{c24:eq:G-entry-cases}
(\mathsf{G}_N)_{i,j}=
\begin{cases}
1,&i=j=0,\\
3,&i=j=1,\\
a+3,&i=j\ge2,\\
1,&|i-j|=1\text{ and }\min\{i,j\}=0,\\
a+2,&|i-j|=1\text{ and }\min\{i,j\}\ge1,\\
1,&|i-j|=2,\\
0,&|i-j|\ge3.
\end{cases}
\end{equation}
\end{corollary}

\begin{proof}
The determinant identity \cref{c24:eq:h-det-G} follows from
\cref{c24:lem:r-monic,lem:monic-change-basis}.  
By \cref{c24:eq:gram-kernel-final},
\[K_a(u,v)=\Pi_a(u,v)\sum_{k\ge0}(uv)^k.
\]
The monomial \(1\) in \(\Pi_a\) contributes \(1\) to every diagonal entry.  The monomial \(2uv\) contributes an additional \(2\) to every diagonal entry with index at least \(1\), and \(au^2v^2\) contributes an additional \(a\) to every diagonal entry with index at least \(2\). This gives the three diagonal cases in \cref{c24:eq:G-entry-cases}.

The monomials \(u\) and \(v\) contribute \(1\) to every first off-diagonal entry.  The monomials
\((a+1)u^2v\) and \((a+1)uv^2\) contribute an additional \(a+1\) to first off-diagonal entries whose smaller index is at least \(1\). Hence the first off-diagonal is \(1\) at the upper-left boundary and \(a+2\) thereafter.  The monomials \(u^2\) and \(v^2\) contribute \(1\) to every second off-diagonal entry.  Since \(\Pi_a\) has no monomial whose exponent difference exceeds \(2\), all entries at distance at least \(3\) vanish.  This proves \cref{c24:eq:G-entry-cases} and hence \cref{c24:eq:G-matrix-display}.
\end{proof}

\subsection{Rank-one decomposition}\label{c24:sec:rank-one}

Set $\lambda=a+1$.
For \(N\ge1\), let \(S_N\) be the lower shift matrix, indexed from
\(0\) to \(N-1\):
\[(S_N)_{i,j}=
\begin{cases}
1,&i=j+1,\\
0,&\text{otherwise}.
\end{cases}
\]
Thus \(S_N^N=0\).  Define
\begin{equation}\label{c24:eq:L-U-def}
  L_N=I_N+\lambda S_N+S_N^2,
  \qquad
  U_N=I_N+S_N^{\trans}+(S_N^{\trans})^2,
\end{equation}
and put $B_N=L_NU_N$. Both \(L_N\) and \(U_N\) are unit triangular, so
\begin{equation}\label{c24:eq:det-B-one}
  \det L_N=\det U_N=\det B_N=1.
\end{equation}

Let \(e_0,e_1,\ldots,e_{N-1}\) denote the standard basis vectors.
For \(N\ge2\), define $w_N=-a e_0+(1-a)e_1$.

\begin{proposition}[Rank-one decomposition]
\label{c24:prop:rank-one-decomp}
For every \(N\ge2\),
\begin{equation}\label{c24:eq:rank-one-decomp}
  \mathsf{G}_N=B_N+e_1w_N^{\trans}.
\end{equation}
\end{proposition}
\begin{proof}
Let $\ell_0=1$, $\ell_1=\lambda$, $\ell_2=1$, and $p_0=p_1=p_2=1$,
with \(\ell_k=p_k=0\) outside \(\{0,1,2\}\).  From \cref{c24:eq:L-U-def}, we have 
$(L_N)_{i,k}=\ell_{i-k}$, $(U_N)_{k,j}=p_{j-k}$. Therefore
\begin{equation}\label{c24:eq:B-entry-convolution}
(B_N)_{i,j}=\sum_{k=0}^{N-1}\ell_{i-k}p_{j-k}.
\end{equation}

For \(i=j\ge2\), the only nonzero summands in
\cref{c24:eq:B-entry-convolution} correspond to
\(k=i,i-1,i-2\), and hence $(B_N)_{i,i}=\lambda+2=a+3$.
For \(i=j+1\) with \(j\ge1\), the two nonzero summands give $(B_N)_{j+1,j}=\lambda+1=a+2$, and the same calculation gives \((B_N)_{j,j+1}=a+2\).  At distance \(2\), there is exactly one nonzero summand, so $(B_N)_{i,j}=1$ when $|i-j|=2$. At distance at least \(3\), there is no admissible \(k\), so the entry is zero.

At the upper-left boundary, direct evaluation of
\cref{c24:eq:B-entry-convolution} gives
\[
  (B_N)_{0,0}=1,
  \quad
  (B_N)_{0,1}=1,
  \quad
  (B_N)_{0,2}=1,
\]
whenever the indicated columns exist, and
\[
  (B_N)_{1,0}=\lambda,
  \qquad
  (B_N)_{1,1}=\lambda+1,
  \qquad
  (B_N)_{1,2}=\lambda+1,
  \qquad
  (B_N)_{1,3}=1,
\]
whenever the indicated columns exist.  Comparing with
\cref{c24:eq:G-entry-cases}, every entry of \(B_N\) agrees with the
corresponding entry of \(\mathsf{G}_N\), except
\[
  (\mathsf{G}_N)_{1,0}-(B_N)_{1,0}
  =1-\lambda=-a
\]
and
\[
  (\mathsf{G}_N)_{1,1}-(B_N)_{1,1}
  =3-(\lambda+1)=1-a.
\]
Thus the difference \(\mathsf{G}_N-B_N\) has only one nonzero row,
the row with index \(1\), and that row is $(-a,1-a,0,\ldots,0)=w_N^{\trans}$.
This is exactly \cref{c24:eq:rank-one-decomp}.
\end{proof}

We shall use the following standard determinant identity.

\begin{lemma}\label{c24:lem:matrix-det}
Let \(B\) be an invertible \(N\times N\) matrix over a commutative ring, and let \(u,v\) be column vectors of length \(N\).  Then
\begin{equation}\label{c24:eq:matrix-det-lemma}
\det(B+uv^{\trans})=\det(B)\bigl(1+v^{\trans}B^{-1}u\bigr).
\end{equation}
\end{lemma}
\begin{proof}
Factor  $B+uv^{\trans}=B\bigl(I_N+B^{-1}uv^{\trans}\bigr)$.
It is therefore enough to prove $\det(I_N+xy^{\trans})=1+y^{\trans}x$.
This identity follows by multilinearity of the determinant in the columns.  The \(j\)-th column of \(I_N+xy^{\trans}\) is
\(e_j+y_jx\).  In the multilinear expansion, the term obtained by choosing every \(e_j\) is \(1\).  Any term in which \(x\) is chosen from two or more columns vanishes because it has two proportional columns.  The terms in which \(x\) is chosen from exactly one column sum to $\sum_{j=0}^{N-1}y_jx_j=y^{\trans}x$. This proves \cref{c24:eq:matrix-det-lemma}.
\end{proof}

To compute the scalar in \cref{c24:eq:matrix-det-lemma}, define sequences
\((\alpha_n)_{n\ge0}\) and \((\beta_n)_{n\ge0}\) by
\begin{equation}\label{c24:eq:alpha-beta-gf}
  \frac{1}{1+z+z^2}=\sum_{n\ge0}\alpha_nz^n,
  \qquad
  \frac{1}{1+\lambda z+z^2}=\sum_{n\ge0}\beta_nz^n.
\end{equation}

\begin{lemma}\label{c24:lem:toeplitz-inverse}
For every \(N\ge1\), we have 
\begin{align}
  U_N^{-1}&=\sum_{k=0}^{N-1}\alpha_k(S_N^{\trans})^k,\label{c24:eq:U-inverse}\\
  L_N^{-1}&=\sum_{k=0}^{N-1}\beta_kS_N^k.\label{c24:eq:L-inverse}
\end{align}
\end{lemma}
\begin{proof}
The first identity in \cref{c24:eq:alpha-beta-gf} means that
$(1+z+z^2)\sum_{k\ge0}\alpha_kz^k=1$.
Substitute the nilpotent matrix \(S_N^{\trans}\) for \(z\).  Since \((S_N^{\trans})^N=0\), all terms of degree at least \(N\) vanish, and we obtain
\[\bigl(I_N+S_N^{\trans}+(S_N^{\trans})^2\bigr)\left(\sum_{k=0}^{N-1}\alpha_k(S_N^{\trans})^k\right)=I_N.
\]
This proves \cref{c24:eq:U-inverse}.  The proof of \cref{c24:eq:L-inverse} is identical, using the second reciprocal series
in \cref{c24:eq:alpha-beta-gf} and substituting \(S_N\) for \(z\).
\end{proof}

\begin{proposition}\label{c24:prop:h-partial-sum}
For every \(n\ge0\), we have 
\begin{equation}\label{c24:eq:h-partial-sum}
 h_n=1+\sum_{k=0}^{n-1}\bigl((1-a)\alpha_k-a\alpha_{k+1}\bigr)\beta_k,
\end{equation}
where the sum is empty when \(n=0\).
\end{proposition}
\begin{proof}
For \(n=0\), \(h_0=\mu_0=1\), so the formula is immediate.  Let \(n\ge1\), and set \(N=n+1\).  By
\cref{c24:eq:h-det-G,c24:eq:rank-one-decomp,c24:eq:det-B-one} and the matrix determinant lemma,
\begin{align}
h_n=\det\mathsf{G}_N=\det(B_N+e_1w_N^{\trans})=1+w_N^{\trans}B_N^{-1}e_1.
\label{c24:eq:h-rank-one-scalar}
\end{align}
Since \(B_N=L_NU_N\), $B_N^{-1}=U_N^{-1}L_N^{-1}$.
Let $z=L_N^{-1}e_1$.
By \cref{c24:eq:L-inverse}, $z=\sum_{k=0}^{N-1}\beta_kS_N^ke_1$.
Now \(S_N^ke_1=e_{k+1}\) for \(0\le k\le N-2\), while \(S_N^{N-1}e_1=0\).  Therefore
\begin{equation}\label{c24:eq:z-components}
  z_0=0,\qquad z_j=\beta_{j-1} \quad(1\le j\le N-1).
\end{equation}
Set $y=U_N^{-1}z=B_N^{-1}e_1$.
By \cref{c24:eq:U-inverse}, the entries of \(U_N^{-1}\) are
\[(U_N^{-1})_{i,j}=
\begin{cases}
\alpha_{j-i},&j\ge i,\\
0,&j<i.
\end{cases}
\]
Using \cref{c24:eq:z-components}, we obtain
\begin{align*}
y_0=\sum_{j=1}^{N-1}\alpha_j\beta_{j-1},\quad 
y_1 =\sum_{j=1}^{N-1}\alpha_{j-1}\beta_{j-1} =\sum_{k=0}^{N-2}\alpha_k\beta_k.
\end{align*}
Since \(w_N=-ae_0+(1-a)e_1\), we have 
\begin{align*}
w_N^{\trans}y&=-ay_0+(1-a)y_1=-a\sum_{j=1}^{N-1}\alpha_j\beta_{j-1}+(1-a)\sum_{k=0}^{N-2}\alpha_k\beta_k
\\&=\sum_{k=0}^{N-2}\bigl((1-a)\alpha_k-a\alpha_{k+1}\bigr)\beta_k.
\end{align*}
Because \(N-2=n-1\), substitution into \cref{c24:eq:h-rank-one-scalar} proves \cref{c24:eq:h-partial-sum}.
\end{proof}

\subsection{The Hankel generating function}\label{c24:sec:recurrence}

Define
\begin{equation}\label{c24:eq:gamma-def}
  \gamma_n=\bigl((1-a)\alpha_n-a\alpha_{n+1}\bigr)\beta_n\qquad(n\ge0),
\end{equation}
and let $\Gamma(x)=\sum_{n\ge0}\gamma_nx^n$.
By \cref{c24:eq:h-partial-sum}, we have $h_n=1+\sum_{k=0}^{n-1}\gamma_k$.
Consequently,
\begin{align}
H(x):=\sum_{n\ge0}h_nx^n=1+\sum_{n\ge1}x^n+\sum_{n\ge1}\sum_{k=0}^{n-1}\gamma_kx^n
=\frac{1}{1-x}+\frac{x}{1-x}\sum_{k\ge0}\gamma_kx^k
=\frac{1+x\Gamma(x)}{1-x}.
\label{c24:eq:H-from-Gamma}
\end{align}
The interchange of summations is valid in the formal-power-series
sense because the coefficient of each fixed power of \(x\) involves
only finitely many \(\gamma_k\).
It remains to determine \(\Gamma(x)\).

From \cref{c24:eq:alpha-beta-gf}, the sequences \(\alpha_n\) and
\(\beta_n\) satisfy
\begin{align}
  \alpha_{n+2}+\alpha_{n+1}+\alpha_n&=0,
  \label{c24:eq:alpha-rec}\\
  \beta_{n+2}+\lambda\beta_{n+1}+\beta_n&=0,
  \label{c24:eq:beta-rec}
\end{align}
where \(\lambda=a+1\).

\begin{lemma}\label{c24:lem:product-recurrence}
The sequence \((\gamma_n)_{n\ge0}\) satisfies
\begin{equation}\label{c24:eq:gamma-recurrence}
\gamma_{n+4}-\lambda\gamma_{n+3}+(\lambda^2-1)\gamma_{n+2}-\lambda\gamma_{n+1}+\gamma_n=0
\end{equation}
for every \(n\ge0\).
\end{lemma}
\begin{proof}
Define the two vectors
\[u_n=\begin{pmatrix}
    \alpha_{n+1}\\
    \alpha_n
\end{pmatrix},
  \qquad
 v_n=
  \begin{pmatrix}
    \beta_{n+1}\\
    \beta_n
  \end{pmatrix}.
\]
The recurrences \cref{c24:eq:alpha-rec,c24:eq:beta-rec} are equivalent to $u_{n+1}=Pu_n$, $v_{n+1}=Qv_n$, where
\[P=\begin{pmatrix}
    -1&-1\\
    1&0
\end{pmatrix},
\qquad
Q=\begin{pmatrix} -\lambda&-1\\ 1&0 \end{pmatrix}.
\]
Let $z_n=u_n\otimes v_n$.
With the coordinate order
\[z_n=
  \begin{pmatrix}
  \alpha_{n+1}\beta_{n+1}\\
  \alpha_{n+1}\beta_n\\
  \alpha_n\beta_{n+1}\\
  \alpha_n\beta_n
  \end{pmatrix},
\]
we have
\[
  z_{n+1}=Tz_n,
  \qquad
  T=P\otimes Q
  =
  \begin{pmatrix}
    \lambda&1&\lambda&1\\
    -1&0&-1&0\\
    -\lambda&-1&0&0\\
    1&0&0&0
  \end{pmatrix}.
\]
Moreover, \cref{c24:eq:gamma-def} can be written as
\begin{equation}\label{c24:eq:gamma-linear-functional}
  \gamma_n
  =
  \begin{pmatrix}0&-a&0&1-a\end{pmatrix}z_n.
\end{equation}

We now compute the characteristic polynomial of \(T\).  For an
indeterminate \(t\),
\[
  tI_4-T
  =
  \begin{pmatrix}
    t-\lambda&-1&-\lambda&-1\\
    1&t&1&0\\
    \lambda&1&t&0\\
    -1&0&0&t
  \end{pmatrix}.
\]
We obtain 
\begin{align*}
\det(tI_4-T)=(1-t^2)+t(t^3-\lambda t^2+\lambda^2t-\lambda)=t^4-\lambda t^3+(\lambda^2-1)t^2-\lambda t+1.
\end{align*}
By the Cayley--Hamilton theorem, $T^4-\lambda T^3+(\lambda^2-1)T^2-\lambda T+I_4=0$.
Multiplying this identity by \(z_n\), using \(z_{n+j}=T^jz_n\), and then applying the row vector in \cref{c24:eq:gamma-linear-functional}, gives exactly \cref{c24:eq:gamma-recurrence}.
\end{proof}

We now determine the initial values of \(\gamma_n\).  Since $\frac{1}{1+z+z^2}=\frac{1-z}{1-z^3}$,
we have
\begin{equation}\label{c24:eq:alpha-initial}
  \alpha_0=1,
  \quad
  \alpha_1=-1,
  \quad
  \alpha_2=0,
  \quad
  \alpha_3=1,
  \quad
  \alpha_4=-1.
\end{equation}
From the second identity in \cref{c24:eq:alpha-beta-gf},
\begin{equation}\label{c24:eq:beta-initial}
  \beta_0=1,
  \quad
  \beta_1=-\lambda,
  \quad
  \beta_2=\lambda^2-1,
  \quad
  \beta_3=-\lambda^3+2\lambda.
\end{equation}
Using \cref{c24:eq:gamma-def,c24:eq:alpha-initial,c24:eq:beta-initial}, we obtain
\begin{align*}
\gamma_0&=\bigl((1-a)\alpha_0-a\alpha_1\bigr)\beta_0=1,\\
\gamma_1&=\bigl((1-a)\alpha_1-a\alpha_2\bigr)\beta_1=(a-1)(-\lambda)=1-a^2,\\
\gamma_2&=\bigl((1-a)\alpha_2-a\alpha_3\bigr)\beta_2=-a(\lambda^2-1)=-a^2(a+2),\\
\gamma_3&=\bigl((1-a)\alpha_3-a\alpha_4\bigr)\beta_3=-\lambda^3+2\lambda=-a^3-3a^2-a+1.
\end{align*}

Let
\begin{equation}\label{c24:eq:D-def}
  D(x)
  =1-\lambda x+(\lambda^2-1)x^2-\lambda x^3+x^4.
\end{equation}
By \cref{c24:eq:gamma-recurrence}, the coefficient of \(x^n\) in
\(D(x)\Gamma(x)\) is zero for every \(n\ge4\).  The coefficients of
degrees \(0,1,2,3\) are therefore
\begin{align*}
[x^0]D(x)\Gamma(x)&=\gamma_0=1,\\
[x^1]D(x)\Gamma(x)&=\gamma_1-\lambda\gamma_0=1-a^2-(a+1)=-a(a+1)=-a\lambda,\\
[x^2]D(x)\Gamma(x)&=\gamma_2-\lambda\gamma_1+(\lambda^2-1)\gamma_0=a-1\\
[x^3]D(x)\Gamma(x)&=0.
\end{align*}
Consequently,
\begin{equation}\label{c24:eq:Gamma-rational}
\Gamma(x)=\frac{1-a(a+1)x+(a-1)x^2}{1-(a+1)x+\bigl((a+1)^2-1\bigr)x^2-(a+1)x^3+x^4}.
\end{equation}

We can now complete the proof of the main theorem.

\begin{proof}[Proof of \cref{thm:c24}]
Insert \cref{c24:eq:Gamma-rational} into \cref{c24:eq:H-from-Gamma}.  With
\(\lambda=a+1\) and \(D(x)\) as in \cref{c24:eq:D-def}, we obtain
\begin{align*}
H(x)=\frac{1}{1-x}\left(1+x\frac{1-a\lambda x+(a-1)x^2}{D(x)}\right)
=\frac{D(x)+x-a\lambda x^2+(a-1)x^3}{(1-x)D(x)}.
\end{align*}
According to 
$$D(x)+x-a\lambda x^2+(a-1)x^3=1-ax+ax^2-2x^3+x^4=(1-x)\bigl(1+(1-a)x+x^2-x^3\bigr),$$
we have
\[H(x)=\frac{1+(1-a)x+x^2-x^3}{1-(a+1)x+\bigl((a+1)^2-1\bigr)x^2-(a+1)x^3+x^4}.\]
This is \cref{c24:eq:main-result}.

Multiplying both sides of \cref{c24:eq:main-result} by its denominator and comparing coefficients gives the recurrence
\cref{c24:eq:h-recurrence-main}.  The coefficients of degrees \(0,1,2,3\) give \cref{c24:eq:h0123-main}.
\end{proof}

\section{The family $(1-(r-2)x+x^2)/(1-sx-x^2)$}\label{sec:c27}

Throughout this section, $R=\mathbb{Q}[r,s]$.

\subsection{The fundamental functional relation}
\label{c27:sec:specialized-moments}

We now specialize \cref{prop:C-closed-form} to the rational
function $g_{r,s}(x)=\frac{1-(r-2)x+x^2}{1-sx-x^2}$.
Set $M(z)=\Ctrans(g_{r,s})(z)=\sum_{n\ge 0}\mu_nz^n$.

\begin{proposition}{\em \cite[Section 9]{Barry2020}}\label{c27:prop:specialized-C-transform}
The series \(M(z)\) has the closed form
\begin{equation}\label{c27:eq:M-closed-form}
 M(z)=\frac{\sqrt{1-4z}-sz}{(1-rz)\sqrt{1-4z}}.
\end{equation}
\end{proposition}

Let $\Lfunc:R[X]\longrightarrow R$, $\Lfunc(X^n)=\mu_n$.
The central-binomial functional \(\Afunc\) is the one defined in \cref{def:central-binomial-functional}.

\begin{proposition}
\label{c27:prop:functional-relation}
For every polynomial \(f(X)\in\Ring[X]\),
\begin{equation}\label{c27:eq:functional-relation}
\Lfunc\bigl((X-r)f(X)\bigr)=-s\Afunc(f(X)).
\end{equation}
\end{proposition}
\begin{proof}
From \cref{c27:eq:M-closed-form},
\begin{equation}\label{c27:eq:M-times-linear-factor}
 (1-rz)M(z)=1-\frac{sz}{\sqrt{1-4z}}.
\end{equation}
Using $M(z)=\sum_{n\ge 0}\mu_nz^n$ 
and \cref{eq:central-binomial-generating-function}, the left-hand side of
\cref{c27:eq:M-times-linear-factor} is
\begin{align*}
 (1-rz)M(z)=\sum_{n\ge 0}\mu_nz^n-r\sum_{n\ge 0}\mu_nz^{n+1}=\mu_0+\sum_{n\ge 1}(\mu_n-r\mu_{n-1})z^n,
\end{align*}
whereas the right-hand side is
\[1-s\sum_{n\ge 0}\binom{2n}{n}z^{n+1}=1-s\sum_{n\ge 1}\binom{2n-2}{n-1}z^n.
\]
Comparing constant terms gives \(\mu_0=1\), and comparing the coefficient of
\(z^n\) for \(n\ge 1\) gives $\mu_n-r\mu_{n-1}=-s\binom{2n-2}{n-1}$.
Replacing \(n\) by \(k+1\), we obtain, for every \(k\ge 0\),
\begin{align*}
\Lfunc\bigl((X-r)X^k\bigr)=\Lfunc(X^{k+1})-r\Lfunc(X^k)=
 \mu_{k+1}-r\mu_k=-s\binom{2k}{k}=-s\Afunc(X^k).
\end{align*}
Since the monomials form an \(\Ring\)-basis of \(\Ring[X]\), linearity proves
\cref{c27:eq:functional-relation} for every polynomial \(f\).
\end{proof}
\begin{lemma}\label{c27:lem:X-r-Gram-A}
For \(i,j\ge 0\), define $J_{i,j}=\Afunc\bigl((X-r)q_iq_j\bigr)$.
Then
\begin{equation}\label{c27:eq:J-entry-cases}
 J_{i,j}
 =
 \begin{cases}
 2-r,&i=j=0,\\
 2(2-r),&i=j\ge 1,\\
 2,&|i-j|=1,\\
 0,&|i-j|\ge 2.
 \end{cases}
\end{equation}
\end{lemma}

\begin{proof}
For \(i=0\), \cref{eq:Xr-q0} gives $(X-r)q_0=q_1+(2-r)q_0$.
Taking the \(\Afunc\)-inner product with \(q_0\) and using
\cref{eq:q-orthogonality}, $J_{0,0}=\Afunc(q_1q_0)+(2-r)\Afunc(q_0^2)=2-r$.
Taking the inner product with \(q_1\), $J_{0,1}=\Afunc(q_1^2)+(2-r)\Afunc(q_0q_1)=2$.
If \(j\ge 2\), orthogonality gives \(J_{0,j}=0\).

For \(i=1\), \cref{eq:Xr-q1} gives $(X-r)q_1=q_2+(2-r)q_1+2q_0$.
Thus $J_{1,1}=2(2-r)$, $J_{1,2}=2$, $J_{1,j}=0$, ($j\ge 3$),
and also \(J_{1,0}=2\), in agreement with symmetry.

For \(i\ge 2\), \cref{eq:Xr-qn} gives $(X-r)q_i=q_{i+1}+(2-r)q_i+q_{i-1}$.
Taking the \(\Afunc\)-inner product with \(q_j\) and using \cref{eq:q-orthogonality}, only the cases \(j=i-1,i,i+1\) survive. 
The resulting values are $J_{i,i-1}=2$, $J_{i,i}=2(2-r)$, and $J_{i,i+1}=2$.
All remaining entries vanish. This proves \cref{c27:eq:J-entry-cases}.
\end{proof}

\subsection{Tridiagonalization of the Gram matrix}\label{c27:sec:tridiagonalization}

Define a sequence of monic polynomials by
\begin{equation}\label{c27:eq:p-basis-definition}
 p_0(X)=1,
 \qquad
 p_n(X)=(X-r)q_{n-1}(X)
 \quad(n\ge 1).
\end{equation}
Because \(q_{n-1}\) is monic of degree \(n-1\), the polynomial \(p_n\) is
monic of degree \(n\).

For \(i,j\ge 0\), set
\begin{equation}\label{c27:eq:Gamma-entry-definition}
 \Gamma_{i,j}=\Lfunc(p_ip_j).
\end{equation}

\begin{proposition}\label{c27:prop:tridiagonal-Gram}
The entries in \cref{c27:eq:Gamma-entry-definition} are
\begin{equation}\label{c27:eq:Gamma-entry-cases}
 \Gamma_{i,j}
 =
 \begin{cases}
 1,&i=j=0,\\
 -s,&\{i,j\}=\{0,1\},\\
 s(r-2),&i=j=1,\\
 2s(r-2),&i=j\ge 2,\\
 -2s,&|i-j|=1\text{ and }\min\{i,j\}\ge 1,\\
 0,&\text{otherwise}.
 \end{cases}
\end{equation}
Consequently, for \(n\ge 2\),
\begingroup
\small
\setlength{\arraycolsep}{3pt}
\begin{equation}\label{c27:eq:Gamma-matrix-explicit}
 \Gamma_n:=\bigl(\Gamma_{i,j}\bigr)_{0\le i,j\le n}
 =
 \begin{pmatrix}
 1&-s&0&0&\cdots&0\\
 -s&s(r-2)&-2s&0&\cdots&0\\
 0&-2s&2s(r-2)&-2s&\ddots&\vdots\\
 0&0&-2s&2s(r-2)&\ddots&0\\
 \vdots&\vdots&\ddots&\ddots&\ddots&-2s\\
 0&0&\cdots&0&-2s&2s(r-2)
 \end{pmatrix}.
\end{equation}
\endgroup
Moreover, $h_n=\det\Gamma_n$ ($n\ge 0$).
\end{proposition}

\begin{proof}
First, $\Gamma_{0,0}=\Lfunc(1)=\mu_0=1$.
For \(j\ge 1\), \cref{c27:eq:p-basis-definition} and
\cref{c27:eq:functional-relation} give
\begin{align}
\Gamma_{0,j}=\Lfunc(p_j)=\Lfunc\bigl((X-r)q_{j-1}\bigr)
=-s\Afunc(q_{j-1}).
 \label{c27:eq:Gamma-zero-row}
\end{align}
By \cref{eq:A-q-values}, we have $\Gamma_{0,1}=-s$, $\Gamma_{0,j}=0$, ($j\ge 2$).
Symmetry gives the corresponding first-column entries.

Now let \(i,j\ge 1\). Then
\begin{align*}
\Gamma_{i,j}=\Lfunc\bigl((X-r)^2q_{i-1}q_{j-1}\bigr)= \Lfunc\left((X-r)\bigl((X-r)q_{i-1}q_{j-1}\bigr)\right) =-s\Afunc\bigl((X-r)q_{i-1}q_{j-1}\bigr),
\end{align*}
By \cref{c27:lem:X-r-Gram-A}, we have 
\begin{align*}
 \Gamma_{1,1}
 &=-s(2-r)=s(r-2),\\
 \Gamma_{i,i}
 &=-2s(2-r)=2s(r-2)
 \qquad(i\ge 2),\\
 \Gamma_{i,i+1}
 &=\Gamma_{i+1,i}=-2s
 \qquad(i\ge 1),
\end{align*}
and \(\Gamma_{i,j}=0\) whenever \(i,j\ge 1\) and \(|i-j|\ge 2\).
Together with \cref{c27:eq:Gamma-zero-row}, this proves
\cref{c27:eq:Gamma-entry-cases,c27:eq:Gamma-matrix-explicit}.

Finally, the polynomials in \cref{c27:eq:p-basis-definition} are monic with
\(\deg p_j=j\). Therefore \cref{lem:monic-change-basis} applies and gives $h_n=\det\Gamma_n$.
\end{proof}

\subsection{The Hankel generating function}
\label{c27:sec:determinants}

We now compute the leading principal minors of
\cref{c27:eq:Gamma-matrix-explicit}.
For \(n=0\),
\begin{equation}\label{c27:eq:h0-proof}
 h_0=\det(1)=1.
\end{equation}
For \(n=1\),
\begin{align}
 h_1=\det
 \begin{pmatrix}
 1&-s\\
 -s&s(r-2)
 \end{pmatrix}
=s(r-2)-s^2
=s(r-s-2).
\label{c27:eq:h1-proof}
\end{align}

For \(n\ge 2\), the last diagonal entry of \(\Gamma_n\) is $d=2s(r-2)$, and the last subdiagonal entry is $e=-2s$.
We now derive the recurrence without omitting the cofactor signs. The matrix \(\Gamma_n\) has size \(n+1\). Expanding its determinant along the last row, the last diagonal entry contributes $d\det\Gamma_{n-1}=dh_{n-1}$.
The only other nonzero entry in the last row is \(e\), located in the
penultimate column. Its cofactor sign is $(-1)^{(n+1)+n}=-1$ when rows and columns are numbered from \(1\) to \(n+1\). After deleting the last row and the penultimate column, the resulting \(n\times n\) minor has, in its last column, exactly one nonzero entry, namely \(e\), in its last row. Expanding that minor along its last column produces $e\det\Gamma_{n-2}=eh_{n-2}$
with positive cofactor sign. Hence the total contribution of the penultimate-column entry in the original last-row expansion is
$e\cdot(-1)\cdot e h_{n-2}=-e^2h_{n-2}$.
Therefore
\begin{equation}\label{c27:eq:h-recurrence-proof}
 h_n=dh_{n-1}-e^2h_{n-2}=2s(r-2)h_{n-1}-4s^2h_{n-2}\qquad(n\ge 2).
\end{equation}
This proves the recurrence asserted in \cref{c27:eq:h-recurrence-intro}.

\begin{proof}[Proof of \cref{thm:c27}]
Let $H(x)=\sum_{n\ge 0}h_nx^n$.
Multiply \cref{c27:eq:h-recurrence-proof} by \(x^n\) and sum over
\(n\ge 2\):
\begin{align}
 \sum_{n\ge 2}h_nx^n=2s(r-2)\sum_{n\ge 2}h_{n-1}x^n-4s^2\sum_{n\ge 2}h_{n-2}x^n.
 \label{c27:eq:sum-recurrence}
\end{align}
Each sum can be expressed in terms of \(H(x)\):
\begin{align*}
 \sum_{n\ge 2}h_nx^n=H(x)-h_0-h_1x,\quad 
 \sum_{n\ge 2}h_{n-1}x^n=x\bigl(H(x)-h_0\bigr),\quad
 \sum_{n\ge 2}h_{n-2}x^n=x^2H(x).
\end{align*}
Substituting these identities into \cref{c27:eq:sum-recurrence} gives
\begin{align*}
 \bigl(1-2s(r-2)x+4s^2x^2\bigr)H(x)=h_0+\bigl(h_1-2s(r-2)h_0\bigr)x.
\end{align*}
By \cref{c27:eq:h0-proof,c27:eq:h1-proof}, we have 
\[H(x) =\frac{1-s(r+s-2)x}{1+2s(2-r)x+4s^2x^2}.\]
This is exactly \cref{c27:eq:main-generating-function}, and the proof of \cref{thm:c27} is complete.
\end{proof}

\begin{example}
For \(r=2\) and \(s=-1\), \cref{c27:eq:main-generating-function} gives $H(x)=\frac{1-x}{1+4x^2}$.
Its coefficient sequence begins $1,-1,-4,4,16,-16,-64,64,\ldots$, 
which agrees with the example in \cite[Example~28]{Barry2020}.
\end{example}






\noindent
{\small \textbf{Acknowledgments:}}
This work was partially supported by the Natural Science Foundation of Henan Province (Grant No. [262300422649]).

\end{document}